\documentclass[12pt]{amsart}
\usepackage{amsmath, amssymb, amsmath, epsfig, graphics, amscd, ifthen, comment}
\usepackage[all]{xy}
\usepackage[utf8]{inputenc}
\newif\iflogvar 
\logvartrue 

\newtheorem{theorem}{Theorem}[section]

\newtheorem{proposition}[theorem]{Proposition}
\newtheorem{corollary}[theorem]{Corollary}

\newtheorem{example}[theorem]{Example}
\newtheorem{conjecture}[theorem]{Conjecture}

\newtheorem{remark}[theorem]{Remark}
\newtheorem{definition}[theorem]{Definition}

\DeclareMathOperator{\HO}{H}
\newcommand{\dd}{\mathbf{d}}
\DeclareMathOperator{\wt}{wt}
\DeclareMathOperator{\Mst}{\mathfrak{M}}
\DeclareMathOperator{\LL}{\mathbb{L}}
\DeclareMathOperator{\tw}{tw}
\DeclareMathOperator{\sst}{sst}
\DeclareMathOperator{\nilp}{nilp}
\DeclareMathOperator{\ee}{\mathbf{e}}
\DeclareMathOperator{\znilp}{z-nilp}
\DeclareMathOperator{\pt}{pt}
\DeclareMathOperator{\dimvect}{\underline{\dim}}
\DeclareMathOperator{\lmod}{-mod}
\newcommand{\ghilb}[1]{Y_{#1}}
\newcommand{\QT}[1]{\mathcal{A}_{#1}}
\DeclareMathOperator{\gr}{Gr}
\newcommand{\gw}[1]{\gr^{\mathrm{W}}_{#1}}
\newcommand{\HHpts}{\mathcal{P}}
\newcommand{\HHstpts}{\mathcal{P}^{\mathrm st}}
\DeclareMathOperator{\Tors}{\mathfrak{T}ors}
\newcommand{\HHhilb}{\mathcal{H}}
\newcommand{\HHst}{\mathcal{H}^{\mathrm{stack}}}
\DeclareMathOperator{\exc}{ex}
\DeclareMathOperator{\Ad}{Ad}
\newcommand{\nn}{n\mathbf{d}_0}
\newcommand{\FF}{\mathcal{F}}
\newcommand{\slope}{\mu}
\newcommand{\OO}{\mathcal{O}}
\newcommand{\xx}{\mathbf{t}} 
\newcommand{\qq}{q} 
\newcommand{\ncHilb}{\mathrm{nc}\Hilb}

\newcommand{\w}{\mathbf{w}}

\newcommand{\half}{\frac{1}{2}}

\renewcommand{\mod}{\mathop{\rm mod}\nolimits}

\newcommand{\Coh}{\mathop{\rm Coh}\nolimits}
\newcommand{\supp}{\mathop{\rm Supp}\nolimits}
\renewcommand{\phi}{\varphi}

\newcommand{\Hom}{\mathop{\rm Hom}\nolimits}

\newcommand{\Ext}{\mathop{\rm Ext}\nolimits}

\newcommand{\Db}[1]{D^b(#1)}

\newcommand{\Irrep}{{\rm Irrep}}

\newcommand{\SL}{{\rm SL}}
\newcommand{\GL}{{\rm GL}}

\newcommand{\PExp}{{\rm PExp}}

\newcommand{\Hilb}{{\rm Hilb}}

\newcommand{\Tr}{{\rm Tr}}

\newcommand{\Sym}{{\rm Sym}}

\renewcommand{\tilde}{\widetilde}

\newcommand{\CC}{\mathbb{C}}

\newcommand{\D}{\mathcal D}
\newcommand{\N}{\mathbb N}

\newcommand{\C}{\mathbb C}

\newcommand{\Z}{\mathbb Z}

\renewcommand{\P}{\mathbb P}
\newcommand{\Q}{\mathbb Q}

\renewcommand{\sl}{{\mathfrak sl}}
\newcommand{\x}{{\bf x}}

\newcommand{\n}{{\bf n}}
\newcommand{\m}{{\bf m}}

\newcommand{\Rep}{\rm Rep}
\newcommand{\color}[6]{} 
\newcommand{\red}{{\rm red}}
\newcommand{\PPP}{\mathcal{P}}
\newcommand{\QQQ}{\mathcal{Q}}

\usepackage{expl3, fp}
\ExplSyntaxOn
\newcommand{\mymod}[2]{\int_mod:nn{#1}{#2}}  
\ExplSyntaxOff

\usepackage{xcolor}
\definecolor{0}{RGB}{240,240,240}  
\definecolor{1}{RGB}{0,204,204}       
\definecolor{2}{RGB}{0,0,0}       
\usepackage{tikz}
\newcounter{x}
\newcounter{y}
\newcounter{z}

\newcommand\xaxis{210}
\newcommand\yaxis{-30}
\newcommand\zaxis{90}

\newcommand\topside[3]{
  \fill[fill=0, draw=black,shift={(\xaxis:#1)},shift={(\yaxis:#2)},
  shift={(\zaxis:#3)}] (0,0) -- (30:1) -- (0,1) --(150:1)--(0,0);
}

\newcommand\leftside[3]{
  \fill[fill=0, draw=black,shift={(\xaxis:#1)},shift={(\yaxis:#2)},
  shift={(\zaxis:#3)}] (0,0) -- (0,-1) -- (210:1) --(150:1)--(0,0);
}

\newcommand\rightside[3]{
  \fill[fill=0, draw=black,shift={(\xaxis:#1)},shift={(\yaxis:#2)},
  shift={(\zaxis:#3)}] (0,0) -- (30:1) -- (-30:1) --(0,-1)--(0,0);
}

\newcommand\cube[3]{
  \topside{#1}{#2}{#3} \leftside{#1}{#2}{#3} \rightside{#1}{#2}{#3}
}

\newcommand\planepartition[1]{
\setcounter{x}{-1}
\foreach \a in {#1} {
\addtocounter{x}{1}
\setcounter{y}{-1}
\foreach \b in \a {
\addtocounter{y}{1}
\setcounter{z}{-1}
 \foreach \c in {1,...,\b} {
 \addtocounter{z}{1}
 \cube{\value{x}}{\value{y}}{\value{z}}}
        }}}

\newcommand\coltopside[3]{
  \fill[fill=\mymod{\the\numexpr1+\value{z}+\value{x}+\value{y}-1}{3}, draw=black,shift={(\xaxis:#1)},shift={(\yaxis:#2)},
  shift={(\zaxis:#3)}] (0,0) -- (30:1) -- (0,1) --(150:1)--(0,0);
}

\newcommand\colleftside[3]{
  \fill[fill=\mymod{\the\numexpr1+\value{z}+\value{x}+\value{y}-1}{3}, draw=black,shift={(\xaxis:#1)},shift={(\yaxis:#2)},
  shift={(\zaxis:#3)}] (0,0) -- (0,-1) -- (210:1) --(150:1)--(0,0);
}

\newcommand\colrightside[3]{
  \fill[fill=\mymod{\the\numexpr1+\value{z}+\value{x}+\value{y}-1}{3}, draw=black,shift={(\xaxis:#1)},shift={(\yaxis:#2)},
  shift={(\zaxis:#3)}] (0,0) -- (30:1) -- (-30:1) --(0,-1)--(0,0);
}

\newcommand\colcube[3]{
  \coltopside{#1}{#2}{#3} \colleftside{#1}{#2}{#3} \colrightside{#1}{#2}{#3}
}

\newcommand\colorplanepartition[1]{
\setcounter{x}{-1}
\foreach \a in {#1} {
\addtocounter{x}{1}
\setcounter{y}{-1}
\foreach \b in \a {
\addtocounter{y}{1}
\setcounter{z}{-1}
\foreach \c in {1,...,\b} {
\addtocounter{z}{1}
\colcube{\value{x}}{\value{y}}{\value{z}}}}}}
 \usepackage{ytableau}
        
\begin{document}

\title[{Enumerating coloured partitions}]{Enumerating coloured partitions \\ in 2 and 3 dimensions}
\author{Ben Davison}
\address{School of Mathematics, University of Edinburgh}
\email{ben.davison@ed.ac.uk}
\author{Jared Ongaro} 
\address{School of Mathematics, University of Nairobi}
\email{ongaro@uonbi.ac.ke}
\author{Bal\'azs Szendr\H oi}
\address{Mathematical Institute, University of Oxford}
\email{szendroi@maths.ox.ac.uk}
\begin{abstract} We study generating functions of ordinary and plane partitions coloured by the action of a finite subgroup of the corresponding special linear group. After reviewing known results for the case of ordinary partitions, we formulate a conjecture concerning a factorisation property of the generating function of coloured plane partitions that can be thought of as an orbifold analogue of a conjecture of Maulik et al., now a theorem, in three-dimensional Donaldson-Thomas theory. We study natural quantisations of the generating functions arising from geometry, discuss a quantised version of our conjecture, and prove a positivity result for the quantised coloured plane partition function under a geometric assumption.
\end{abstract}

\maketitle
\thispagestyle{empty}

\tableofcontents

\section{Introduction} 

Let $\PPP$ and $\QQQ$ denote the set of all {\em partitions}, respectively {\em plane partitions}. With a dimension shift, these can be defined 
as follows: a partition $\lambda\in\PPP$ is a finite subset $\lambda\subset\N^2$ of the non-negative lattice quadrant, stable under forces 
acting in the negative direction along each of the coordinate axes; a plane partition $\alpha\in\QQQ$ is a finite subset $\alpha\subset\N^3$ of the non-negative lattice octant with the same property (see Figure~\ref{fig:first}). In both cases, we will refer to the lattice points contained in a partition or plane partition as
{\em boxes}. 
There are other well-known equivalent definitions of these sets~\cite{And}, mostly less symmetric than the one given here. Both partitions and plane partitions have a measure of size $|\ |\colon\PPP\to\N$, respectively $|\ |\colon\QQQ\to\N$, given by ``counting the number of boxes''. 
The corresponding generating functions
\[E(t)= \sum_{\lambda\in\PPP}t^{|\lambda|}\]
and
\[ M(t)=\sum_{\alpha\in\QQQ}t^{|\alpha|} 
\]
have product expansions known since the 18th, respectively 19th centuries:
\begin{equation}E(t)= \prod_{m>0}(1-t^m)^{-1}\end{equation}
and
\begin{equation}\label{eq_mcmahon}M(t)=\prod_{m>0}(1-t^m)^{-m}.\end{equation}

\begin{figure}[htbp]

\begin{minipage}{.4\textwidth}
\ytableausetup
{boxsize=1.35em}
\ydiagram[*(0)]{2,3,4,4,6,8,8,10}
\end{minipage}
\hspace{2cm}
\begin{minipage}{.4\textwidth}
\begin{tikzpicture}[scale=0.55]
\planepartition{{5,3,2,2,2,2,1},{5,3,2,2},{2,2,1,1},{2,2},{1},{1}}
\end{tikzpicture}
\end{minipage}
\label{fig:first}
\caption{A partition and a plane partition}
\end{figure}
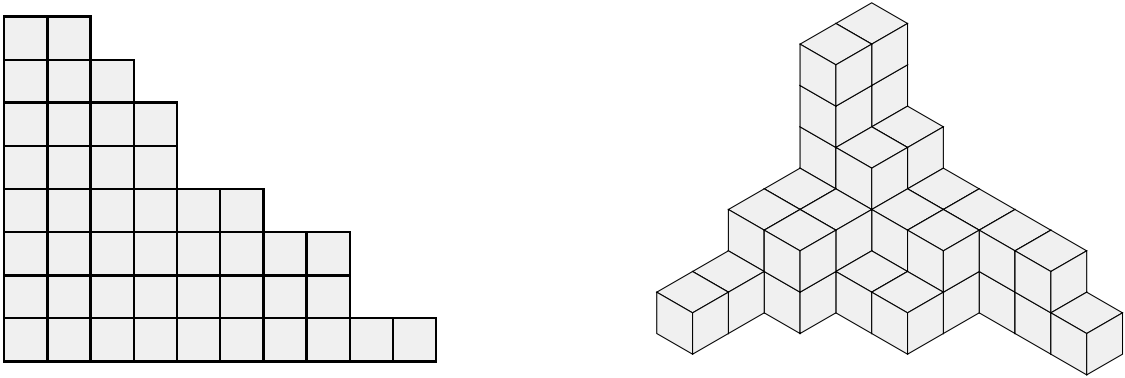

In this paper, we are interested in ``coloured" and ``quantized" versions of these generating functions, where the colours are coming from the action of a finite matrix group $G<\GL_d(\C)$ for $d=2,3$, while the quantization is coming from geometry. 
In fact we will only discuss in detail the case when $G<\SL_d(\C)$; there is 
more structure around in this case, and stronger results possible. 
We start in Section~2 by reviewing the well-known story for $d=2$ and $G<\SL_2(\C)$ abelian;
we then explain what can be said in a non-abelian case. Then we turn our attention to $d=3$, restricting to $G$ abelian. 
In Section 3, we state as Conjecture~\ref{main_conj} a factorisation property of the coloured generating function, that can be thought of as
an orbifold version of the famous conjecture of Maulik, Nekrasov, Okounkov and Pandharipande~\cite{MNOP}, now a theorem of Bridgeland and Toda, 
that states a formally analogous factorisation property of the Donaldson--Thomas partition function attached to a Calabi--Yau threefold (see Remark~\ref{rem:mnop}). 
We check one nontrivial consequence of our conjecture, and present some numerical evidence. We recall from
the literature some cases in which Conjecture~\ref{main_conj} is known, and discuss the structural features of the resulting formulae, using the language of the plethystic exponential. In Section 4, we discuss the relationship of the problem to the geometry of the quotient $\C^3/G$ 
and its minimal resolution(s), leading to a quantized version of the generating function, as well as Conjecture~\ref{qconj}, 
a quantized version of Conjecture~\ref{main_conj}. We close by explaining how Hodge theory 
and the theory of wall-crossing structures gives a framework to attack the quantized conjecture, and proves a related general positivity result,
Theorem~\ref{thm:positivity}, concerning the quantum partition function, under a geometric assumption.

For $d=3$, geometry dictates that the generating functions should include certain signs. In the combinatorial Section 3,
we will ignore these, at the expense of a somewhat unconventional definition of a twisted plethystic exponential $\PExp^\sigma$. 
The signed version of the generating function, defined in the geometric Section 4, correctly accounts for all these signs. 

In our treatment, the relationship between the results for partitions and those for plane partitions is purely formal; there is a very close similarity
between the shapes of the various formulae and their symmetries, but the proofs are different. 
It is hoped that the reader will find the parallels drawn
between the two stories sufficient justification for the joint discussion. 


\medskip

\noindent{\bf Notation} \ Given a set of variables $\xx=(t_0, \ldots, t_{r-1})$ and $\n=(n_0,\ldots, n_{r-1})\in\N^r$, denote $\xx^\n=\prod_i t_i^{n_i}$. 
Given a set of {\em non-negative} integers $a_\n\in\N$ for $\n\in\N^r\setminus\{0\}$ and choices of signs $\sigma_\n\in\{\pm 1\}$ for those $\n$ for which $a_\n\neq 0$, we define the twisted plethystic exponential
\[\PExp^\sigma\left(\sum_{\n\in\N^r\setminus\{0\}}\sigma_\n a_\n \xx^\n\right) = \prod_{\n\in\N^r\setminus\{0\}}(1-\sigma_\n \xx^\n)^{-\sigma_\n a_\n}.\]
Thus, according to this definition, for a single variable $t$,
\[\PExp^\sigma(t) = (1-t)^{-1}\]
but, unusually, 
\[\PExp^\sigma(-t) = 1+t.\]
In fact it is immediate to check that, when expanded around $\xx=0$,
\[\PExp^\sigma\left(\sum_{\n\in\N^r\setminus\{0\}}\sigma_\n a_\n \xx^\n\right) \in \N[[\xx]].\]
This is natural in the context of combinatorial generating functions, but of course destroys all algebraic properties of the ordinary plethystic exponential. With the correct choice of signs, all would be well. Notice in passing that (always expanding around $t=0$)
\[E(t)= \PExp^\sigma\left(\frac{t}{1-t}\right)
\]
and
\[M(t)= \PExp^\sigma\left(\frac{t}{(1-t)^2}\right).
\]
Some of the statements in the case $d=3$ are most conveniently expressed in terms of the functions
\[M(s,t)=\prod_{m>0}(1-st^m)^{-m} = \PExp^\sigma\left(\frac{st}{(1-t)^2}\right)
\]
and
\[\tilde M(s,t)=M(s,t)\cdot M(s^{-1},t) = \PExp^\sigma\left(\frac{(s+s^{-1})\, t}{(1-t)^2}\right);\]
note
\[\tilde M(-s,t)^{-1}= \PExp^\sigma\left(\frac{-(s+s^{-1})\,t}{(1-t)^2}\right).\]

Given a finite group $G$ with $r$ conjugacy classes, let $\Irrep(G)=\{\rho_0={\rm triv}, \rho_1, \ldots, \rho_{r-1}\}$ be the set of irreducible 
representations of $G$. For $G$ cyclic of order $r$, we label so that $\rho_k=\rho_1^{\otimes k}$; for other groups, we will explain 
the labelling explicitly. 

\medskip

\noindent{\bf Acknowledgements}  \ This paper benefitted from conversations and correspondence with Tom Bridgeland, Jim Bryan, Alastair Craw, \'Ad\'am Gyenge, Lotte Hollands, Paul Johnson, Sheldon Katz and Andrew Neitzke. It is partly based on work supported by the National Science Foundation under Grant No~1440140, done while the authors were in residence at the Mathematical Sciences Research Institute in Berkeley, California during the Spring of 2018. Support from the Mentoring African Research in Mathematics grant scheme of the London Mathematical Society, and the International Science Programme of Uppsala Univesity and Sida, is also acknowledged.  During the writing of the paper, BD was supported by the starter grant ``Categorified Donaldson--Thomas theory'' No. 759967 of the European Research Council, and a Royal Society research fellowship.

\section{Partitions}

\subsection{The abelian case: $r$-coloured partitions with diagonal colouring}
\label{diag_colouring}
As is well known, an abelian subgroup $G<\SL_2(\C)$ is necessarily cyclic, of some order $r\geq 1$. 
For concreteness, we fix an isomorphism $G\cong\mu_r$ with the group of $r$th roots of unity in~$\CC$, diagonally embedded in $\SL_2(\C)$ in such a way that the action is given by 
\[
\xi\mapsto \left(\begin{array}{cc} \xi& 0\\ 0&\xi^{-1}\end{array}\right).
\]
We label representations of $G$ so that $\C^2\cong \rho_1\oplus\rho_{r-1}$ as $G$-representations.
We will call this action $\mu_r(1,r-1)$ to record the weights.
The $G$-action extends to an action on the coordinate ring $\C[x,y]$ that we can identify with the semigroup ring $\C[\N^2]$. Each monomial 
$x^iy^j\in \C[\N^2]$ spans a one-dimensional $G$-eigenspace in the representation $\rho_{i-j \mod r}$; this gives a diagonal colouring (labelling) 
of the monomials in $\C[\N^2]$ with characters of $G$; see Figure \ref{fig:pa_nd_pp}.

\begin{figure}[htbp]
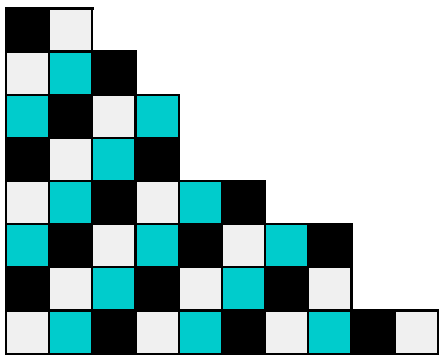

\ytableausetup
{boxsize=1.35em}
 \ytableausetup{nosmalltableaux}
  \begin{ytableau}
   *(2) & *(0)    \\
    *(0) & *(1)  & *(2)  \\
 *(1) &  *(2)  & *(0) & *(1) \\
  *(2) & *(0)  & *(1) & *(2) \\
  *(0) & *(1) & *(2) & *(0)  & *(1) & *(2)  \\
*(1) & *(2) & *(0)  & *(1) & *(2) & *(0) & *(1) & *(2)  \\
 *(2) & *(0)  & *(1) & *(2) & *(0)  & *(1) & *(2) & *(0)\\
 *(0) & *(1) & *(2) & *(0)  &*(1) & *(2) & *(0)  & *(1) & *(2) & *(0) 
  \end{ytableau}

 \caption{The partition from Figure 1 coloured by the group $\mu_3(1,-1)<\SL_2$}
\label{fig:pa_nd_pp}
\end{figure}

Given a partition $\lambda\subset\N^2$, we can define an $r$-tuple of non-negative integers 
\[\w(\lambda) = (w_0(\lambda), \ldots, w_{r-1}(\lambda))\]
counting the number of boxes (monomials) in $\lambda$ of colour $j$ for each $j$. Clearly
\[|\lambda| = \sum_{j=0}^{r-1} w_j(\lambda).
\]
One can then form the generating function
\[ Z^{\mu_r(1,r-1)}(\xx) = \sum_{\lambda\in\PPP}\xx^{\w(\lambda)}
\]
for a set of variables $\xx=(t_0, \ldots, t_{r-1})$. Specialising to $t_i=t$ recovers $E(t)$. The following reasonably standard
result gives a closed formula for this generating function.

\begin{theorem} Setting $t=\prod_{j=0}^{r-1}t_j$ to be the product of the variables, there is a factorisation
\begin{equation}Z^{\mu_r(1,r-1)}(\xx) = E(t)^r \cdot Z_\red^{\mu_r(1,r-1)}(\xx)
\label{2dfactorisation}
\end{equation}
into a power of the function $E(t)$ and a {\em reduced series} $Z_\red^{\mu_r(1,r-1)}\in\N[[\xx]]$. This latter 
series is itself a generating function
\begin{equation} Z_\red^{\mu_r(1,r-1)}(\xx) = \sum_{\lambda\in\PPP^{(r)}}\xx^{\w(\lambda)},
\label{2dreduced}
\end{equation}
where $\PPP^{(r)}\subset\PPP$ is the set of {\em $r$-core partitions}~\cite{JK}. Moreover, it admits the following expression:
\[Z_\red^{\mu_r(1,r-1)}(\xx) = \sum_{\m\in\Lambda_{A_{r-1}}}x^{\frac{\langle \m, \m\rangle}{2}}\prod_{i=1}^{r-1} t_i^{m_i},
\]
where $\Lambda_{A_{r-1}}\cong (\Z^{r-1}, \langle\  \rangle)$ is the type $A_{r-1}$ root (Cartan) lattice, the lattice corresponding to the Dynkin diagram of type
$A_{r-1}$. 
\label{thm_abelian2d}
\end{theorem}

\begin{proof} We recall the outline of the argument. The process of taking the $r$-quotient of a partition~\cite{JK} gives a combinatorial bijection 
\[
\PPP \longleftrightarrow \PPP^r \times \PPP^{(r)}
\] 
between the set of partitions, and $r$-tuples of partitions together with an $r$-core.  
This gives the factorisation formula~\eqref{2dfactorisation}, together with the interpretation~\eqref{2dreduced}. 
The reduced partition function can then be determined with the help of a further bijection
\[
\PPP^{(r)} \longleftrightarrow \Lambda_{A_{r-1}}.
\]
For details, see for example~\cite{GKS}.
\end{proof}
\begin{example}\rm For $r=2$, we are colouring partitions by the checkerboard colouring. The $2$-cores are the staircase partitions. We get the generating function of checkerboard-coloured partitions
\[Z^{\mu_2(1,1)}(t_0,t_1) = E(t_0t_1)^2 \cdot \sum_{m=-\infty}^{\infty} t_0^{m^2} t_1^{m^2+m}.
\]
The sum on the right hand side is essentially the sum side of the Jacobi triple product identity, and so it also admits an infinite product form. 
\end{example}

\subsection{Geometry and Lie theory}

Given a finite subgroup $G<\SL_2(\C)$, the corresponding action on $\C^2$ preserves the two-form $dx\wedge dy$. The
quotient space $X=\C^2/G$ has a rational (du Val, simple) surface singularity at the origin, and is smooth elsewhere. 
It admits a unique Calabi--Yau resolution $\pi\colon Y\to X$. 
The exceptional curves $E_i\subset Y$ of $\pi$ are all rational, and form an intersection configuration that is described by a simply-laced, connected Dynkin
diagram, necessarily of $ADE$ type. For a cyclic $G<\SL_2(\C)$ of size $r$, we get the type $A_{r-1}$ configuration, giving a superficial 
reason for the appearance of the corresponding root lattice above. 

The $2$-dimensional McKay correspondence~\cite{Rei} gives a different appearance of the Dynkin diagram. As we recall in the next section, McKay defined, using the embedding $G<\SL_2(\C)$,
a graph structure on the set $\Irrep(G)=\{\rho_0={\rm triv}, \rho_1, \ldots, \rho_{r-1}\}$. This turns out to give an {\em affine} Dynkin diagram, of the same type, with distinguished node $\rho_0$ whose removal returns the finite Dynkin diagram.

The appearance of the finite and affine Dynkin diagrams indicates a possible role for Lie algebras in this picture. The subjects are indeed
intimately connected; this is a very large area, so 
we only mention the issue relevant for the present discussion. For type $A$, the set of all $r$-coloured partitions gives one combinatorial 
model of the so-called {\em crystal basis}~\cite{Kan} of the basic representation of the affine Lie algebra $\widehat{{\mathfrak gl}}_r$, an infinite-dimensional Lie algebra closely related to the affine root system $\hat A_{r-1}$. From a representation-theoretic point of view, the formulas \eqref{2dfactorisation}-\eqref{2dreduced} compute the character of the basic representation of this Lie
algebra, compatibly with the Frenkel--Kac construction. For details from the present point of view, see \cite[Section 4 and Appendix A]{GyNSz}. 
For the special case $r=1$, $\widehat{{\mathfrak gl}}_1$ is just the infinite dimensional Heisenberg algebra, acting on its standard (fermionic) Fock
space representation. 

\subsection{Partitions and Euler characteristics of Hilbert schemes}\label{sec:2dHilb}
In this subsection we again focus on the cyclic/type A case, with $G\cong\mu_r$ acting with weights $(1,-1)$ as above.
Denote by \[\Hilb(\CC^2)=\sqcup_n\Hilb^n(\CC^2)\] the Hilbert scheme of points of $\CC^2$, whose $\CC$-points consist of finite-codimension ideals of $H^0(\mathcal{O}_{\CC^2})\cong \CC[x,y]$. 
The $G$-fixed locus $\Hilb(\CC^2)^G\subset\Hilb(\CC^2)$ can also be thought of 
as the Hilbert scheme $\Hilb([\CC^2/G])$ attached to the orbifold $[\CC^2/G]$, as 
any invariant ideal is $G$-equivariant in a canonical way. This fixed locus naturally decomposes as
\[\Hilb([\CC^2/G])= \bigsqcup_{\rho\in\Rep(G)}\Hilb^{\rho}([\CC^2/G]),\]
where for a finite-dimensional representation $\rho\in\Rep(G)$, we denote
\[\Hilb^{\rho}([\CC^2/G])=  \{ I \in \Hilb(\CC^2)^G \colon  (\CC[x,y]/I) \simeq_G \rho \}\]
the $\rho$-Hilbert scheme of the orbifold $[\CC^2/G]$, with $\simeq_G$ denoting $G$-equivariant isomorphism.  Each of the spaces in the decomposition is smooth; this is a consequence of the smoothness of $\Hilb(\CC^2)$.  

Given a point in $\Hilb([\CC^2/G])$, the associated $\mathbb{C}[x,y]$-module $\mathcal{F}=(\CC[x,y]/I)$ decomposes canonically into a direct sum of weight spaces $\mathcal{F}=\oplus_{0\leq i\leq r-1}\mathcal{F}_i$, where $\mathcal{F}_i$ is a direct sum of copies of the irreducible representation $\rho_i$.  The action of~$x$ on $\mathcal{F}$ cyclically permutes these summands, sending $\mathcal{F}_i$ to $\mathcal{F}_{i+1}$ (where $i$ is taken modulo $r$) and the action of~$y$ cyclically permutes these summands in the opposite direction.  A $G$-equivariant $\mathbb{C}[x,y]$-representation, then, corresponds to a representation of a quiver with vertices indexed by the irreducible $G$-representations, and arrows in both directions between~$i$ and~$i+1$ (taken modulo $r$).  We label the arrows starting at~$i$ corresponding to the action of~$x$ and~$y$ with labels~$x_i$ and~$y_i$ respectively.  The resulting oriented graph is the \textit{McKay quiver} $Q$ for $G$, and is the double of the McKay graph mentioned above.  This quiver can be defined outside of type A, and for arbitrary $G$-representations: the number of arrows from $i$ to $j$ is defined to be the number of copies of $\rho_j$ inside $V\otimes \rho_i$.  

Coming back to our example of $G<\SL_2(\CC)$, a representation of this quiver corresponds to a $G$-equivariant representation of $\CC[x,y]$ precisely if the commutativity relations $y_{i+1}x_i=x_{i-1}y_i$ hold.  The quotient of the free path algebra of $Q$ by these relations is known as
the {\em preprojective algebra}. 

Let $T$ be the subgroup of diagonal matrices in $\GL_2(\CC)$. Then $T$ acts on $\Hilb(\mathbb{C}^2)$ and this action moreover commutes with the $G$-action, and hence induces an action on each component $\Hilb^{\rho}([\CC^2/G])$; this is where we are making essential use of the assumption that we are in type A.  The fixed points of this action correspond to finite-codimension $\mathbb{Z}^2$-graded ideals $\mathcal{I}\subset \CC[x,y]$.  The monomial basis of $\CC[x,y]/\mathcal{I}$ then defines a coloured partition as in Section \ref{diag_colouring}.  Summing up, for $G=\mu_r(1,r-1)$, there is an equality
\[
Z^{G}(t_0,\ldots,t_{r-1})=\sum_{\dd\in\mathbb{N}^r} \chi(\Hilb^{\rho_{\dd}}([\CC^2/G]))\xx^{\dd}
\]
where $\rho_{\dd}=\bigoplus_{0\leq i \leq r-1}\rho_i^{\oplus d_{i}}$, connecting the geometry of Hilbert schemes with coloured partition counting.

This point of view also allows a deformation (or quantization) of this formula: we can replace Euler characteristics here by the Poincar\'e
polynomials, arriving at the generating functions

\[
Z^G(t_0,\ldots,t_{r-1}; \qq^{1/2})=\sum_{\dd\in\mathbb{N}^r} P(\Hilb^{\rho_{\dd}}([\CC^2/G]), \qq^{1/2})\xx^{\dd}
\]
where for a smooth variety $X$, we denote $P(X,\qq^{1/2}) = \sum_j \dim_\Q H^j(X,\Q)\qq^{j/2}$ its (topological) Poincar\'e polynomial. In the abelian case, 
this function is computed by~\cite{FM} to be
\begin{equation}Z^{\mu_r(1,r-1)}(\xx; \qq^{1/2}) = \left(\prod_{m>0} (1-t^m\qq^{m-1})^{-1}(1-t^m\qq^{m})^{-(r-1)} \right)\cdot Z_\red^{\mu_r(1,r-1)}(\xx).
\label{eq:typeAsurfacerefined}
\end{equation}

We will expand upon the relationship between combinatorics, quiver representations and quantized formulae further 
when we come to plane partitions.

\subsection{A nonabelian family of examples} 

Armed with a full understanding of the cyclic case, we can ask whether there are similar stories to be told for {\em nonabelian} subgroups 
$G<\SL_2(\C)$. Apart from cyclic groups, there is one other infinite set of such subgroups, as well as three sporadic ones. For the infinite series of binary dihedral groups, 
there is a precise analogue of the diagonally coloured partition story. While from the present point of view this is just an amusing 
extension of well-known ideas, we could not resist telling it. For geometric applications of the theory in this section, see~\cite{GyNSz}. 
 
The binary dihedral group\footnote{Another notation for this group would be $BD_{4r-8}$, but we will use this more economical notation.} 
$G=BD_{r}<\SL_2(\C)$ is generated by the matrices 
$\begin{pmatrix} \xi && 0 \\ 0 && \xi^{-1}\end{pmatrix}$ and $\begin{pmatrix} 0 && 1 \\ -1 && 0\end{pmatrix}$,
with $\xi$ a $(2r-4)$-th root of unity. It has order $4r-8$, and $r+1$ conjugary classes. In the correspondences
recalled above, it gives the (finite and affine) Dynkin diagrams of type $D_{r}$. It is natural to label the representations of $BD_{r}$ so 
that $\rho_0={\rm triv}$, $\rho_1$ is the sign representation, whose kernel is the abelian subgroup generated by the first matrix above, $\rho_2$ is
the irreducible representation coming from the embedding $BD_{r}<\SL_2(\C)$, $\rho_3,\ldots, \rho_{r-2}$ are further $2$-dimensional representations (for $r>4$) labelled so that $\rho_j\otimes\rho_2\cong \rho_{j-1}\oplus\rho_{j+1}$ for $j<r-2$, and finally $\rho_{r-1}$ and $\rho_r$ are again one-dimensional, with $\rho_{r-2}\otimes\rho_2\cong \rho_{r-3}\oplus\rho_{r-1}\oplus\rho_r$.

A search for a corresponding partition colouring story begins by recalling that the connection between groups and colourings in the abelian case 
was given by the semigroup ring $\C[\N^2]$ and its decomposition into irreducible $\mu_{r}$-spaces spanned by monomials. 
This does not quite work in the 
dihedral case, but it almost does. It can be checked that for most pairs $(i,j)$, the two-dimensional vector space generated by $x^iy^j$ and 
$x^jy^i$ forms an irreducible representation of $BD_{r}$; for special pairs $(i,j)$, it splits up into two eigenspaces (not generated by monomials). 
We get Figure~\ref{fig_Dn}, first exploited in this language in~\cite{GyNSz}. The diagram
encodes a full description of $\C[\N^2]$ and its decomposition into irreducible $BD_{r}$-spaces. There are $r+1$ different colours 
(labels), some labelling boxes and some half-boxes. Since we wish to study configurations that are equivariant with respect to the dihedral group, 
it is indeed enough to look at the octant of the plane represented on the figure, and then the full configuration is determined by reflection. 
Let us call this combinatorial arrangement the {\em $D_r$-coloured positive octant}.

\begin{figure}
\iflogvar
\centering
\begin{tikzpicture}[scale=0.6, font=\tiny, fill=black!20]
   \foreach \x in {4,5,6,7,8, 9,10,11,12, 13, 14, 15}
                  {
                       \draw (\x,4) -- (\x,4.4);
                       \draw (\x+0.8,5) node {\reflectbox{$\ddots$}};
                  }
  \draw (0,0 ) -- (15.4,0);
    \foreach \y in {1,2,3,4}
        {
          \draw (\y -1,\y ) -- (15.4,\y);
        }
    \foreach \y in {0}
    {
    \foreach \x in {0,1,2,3, 4,5,6,7,8, 9,10,11,12, 13, 14, 15}
                  {
                       \draw (\y+\x ,\y) -- (\y+\x ,\y+1);
                  }
    	\draw (\y+6,\y) -- (\y+5,\y+1);
    	\draw (\y+11,\y) -- (\y+10,\y+1);
    	\draw (\y+1.5,\y+0.5) node {2};
    	\draw (\y+4.5,\y+0.5) node {$r$$-$$2$};
    	\draw (\y+6.5,\y+0.5) node {$r$$-$$2$};
    	\draw (\y+14.5,\y+0.5) node {$r$$-$$2$};
	\draw (\y+9.5,\y+0.5) node {2};
    	\draw (\y+11.5,\y+0.5) node {2};
    	\draw(\y+2.5,\y+0.5) node {$\cdot$};
	\draw(\y+3.5,\y+0.5) node {$\cdot$};
	\draw(\y+7.5,\y+0.5) node {$\cdot$};
	\draw(\y+8.5,\y+0.5) node {$\cdot$};
	\draw(\y+12.5,\y+0.5) node {$\cdot$};
	\draw(\y+13.5,\y+0.5) node {$\cdot$};
	\draw(15.9,\y+0.5) node {\dots};
    }
    \foreach \y in {1}
    {
    \foreach \x in {0,1,2,3,4,5,6,7,8, 9,10,11,12, 13, 14}
                  {
                       \draw (\y+\x ,\y) -- (\y+\x ,\y+1);
                  }
    	\draw (\y+6,\y) -- (\y+5,\y+1);
    	\draw (\y+11,\y) -- (\y+10,\y+1);
    	\draw (\y+1.5,\y+0.5) node {2};
    	\draw (\y+4.5,\y+0.5) node {$r$$-$$2$};
    	\draw (\y+6.5,\y+0.5) node {$r$$-$$2$};
    	\draw (\y+9.5,\y+0.5) node {2};
    	\draw (\y+11.5,\y+0.5) node {2};
    	\draw(\y+3.5,\y+0.5) node {$\cdot$};
	\draw(\y+2.5,\y+0.5) node {$\cdot$};
	\draw(\y+8.5,\y+0.5) node {$\cdot$};
	\draw(\y+7.5,\y+0.5) node {$\cdot$};
	\draw(\y+13.5,\y+0.5) node {$\cdot$};
	\draw(\y+12.5,\y+0.5) node {$\cdot$};
	\draw(15.9,\y+0.5) node {\dots};
    }
    \foreach \y in {2}
    {
    \foreach \x in {0,1,2,3,4,5,6,7,8,9,10,11,12, 13}
                  {
                       \draw (\y+\x ,\y) -- (\y+\x ,\y+1);
                  }
    	\draw (\y+6,\y) -- (\y+5,\y+1);
    	\draw (\y+11,\y) -- (\y+10,\y+1);
    	\draw (\y+1.5,\y+0.5) node {2};
    	\draw (\y+4.5,\y+0.5) node {$r$$-$$2$};
    	\draw (\y+6.5,\y+0.5) node {$r$$-$$2$};
    	\draw (\y+9.5,\y+0.5) node {2};
    	\draw (\y+11.5,\y+0.5) node {2};
    	\draw(\y+3.5,\y+0.5) node {$\cdot$};
	\draw(\y+2.5,\y+0.5) node {$\cdot$};
	\draw(\y+8.5,\y+0.5) node {$\cdot$};
	\draw(\y+7.5,\y+0.5) node {$\cdot$};
	\draw(\y+12.5,\y+0.5) node {$\cdot$};
	\draw(15.9,\y+0.5) node {\dots};
    }

     \foreach \y in {3}
    {
    \foreach \x in {0,1,2,3,4,5,6,7,8, 9,10,11,12}
                  {
                       \draw (\y+\x ,\y) -- (\y+\x ,\y+1);
                  }
    	\draw (\y+6,\y) -- (\y+5,\y+1);
    	\draw (\y+11,\y) -- (\y+10,\y+1);
    	\draw (\y+1.5,\y+0.5) node {2};
    	\draw (\y+4.5,\y+0.5) node {$r$$-$$2$};
    	\draw (\y+6.5,\y+0.5) node {$r$$-$$2$};
    	\draw (\y+9.5,\y+0.5) node {2};
    	\draw (\y+11.5,\y+0.5) node {2};
    	\draw(\y+3.5,\y+0.5) node {$\cdot$};
	\draw(\y+2.5,\y+0.5) node {$\cdot$};
	\draw(\y+8.5,\y+0.5) node {$\cdot$};
	\draw(\y+7.5,\y+0.5) node {$\cdot$};
	\draw(15.9,\y+0.5) node {\dots};
    }

     \foreach \y in {0}
        {
      	    \draw (\y+0.5,\y+0.5) node {0};
        	\draw (\y+5.34,\y+0.25) node  {{\tiny $r$-1}};
        	\draw (\y+5.78,\y+0.75) node {$r$};
        	\draw (\y+10.75,\y+0.65) node {0};
        	\draw (\y+10.25,\y+0.35) node {1};
        }
             \foreach \y in {2}
        {
      	    \draw (\y+0.5,\y+0.5) node {0};
        	\draw (\y+5.34,\y+0.25) node  {{\tiny $r$-1}};
        	\draw (\y+5.78,\y+0.75) node {$r$};
        	\draw (\y+10.75,\y+0.65) node {0};
        	\draw (\y+10.25,\y+0.35) node {1};
        }

        \foreach \y in {1,3}
          {
              	    \draw (\y+0.5,\y+0.5) node {1};
                	\draw (\y+5.25,\y+0.25) node  {$r$};
       	\draw (\y+5.69,\y+0.80) node  {{\tiny $r$-1}};
                	\draw (\y+10.75,\y+0.65) node {1};
                	\draw (\y+10.25,\y+0.35) node {0};
          }
\end{tikzpicture}
\smallskip 
\fi
\caption{The $D_r$-coloured positive octant; single dots in boxes denote a run of numbers between $2$ and $r-2$}
\label{fig_Dn}
\end{figure}

The correct generalisation of partitions can be derived from both the geometric and the representation theoretic points of view. 
One is led to a combinatorial set $\PPP_{D_r}$ of {\em $D_r$-partitions}, finite subsets of the $D_r$-coloured positive octant. 
The precise definition~\cite{KK, GyNSz} is as follows: a finite subset $\beta$ of the {$D_r$-coloured positive octant} 
will be called a {\em $D_r$-partition}, if the following rules are satisfied: 

\begin{enumerate}
\item if a full box is not present in $\beta$, then all full or half boxes to the right of it are also not in $\beta$, and at least one (full or half) box immediately above it is not in $\beta$;
\item if a full or half-box is missing from $\beta$, then all the boxes on the diagonal in the same position, above and to the right of this box, are also missing;
\item if a half box is not in $\beta$, then the full box to the right of it is also not in $\beta$;
\item if both half-boxes sharing the same box position are missing from $\beta$, then the full box immediately above this position is also not
in $\beta$. 
\end{enumerate}

Thus the rules that define these objects are ``partition-like'' at the full boxes; indeed, (1) and (2) (which in the case of ordinary partitions is
superfluous) give the standard definition of ordinary partitions. However, the half-boxes have their own, different rules, leading to interesting
complications in the theory. Figure~\ref{fig:D4} shows some $D_r$-partitions for $r=4$. 

\begin{figure}
\iflogvar
\centering
\begin{tikzpicture}[scale=0.6, font=\footnotesize, fill=black!20]
 \draw (1, 0) -- (5,0);
 \draw (1, 1) -- (5,1);
  \foreach \x in {1,2,3,4,5}
    {
      \draw (\x, 0) -- (\x,1);
    }
  \draw (1.5, 0.5) node {0};    
  \draw (2.5, 0.5) node {2};
  
  \draw (3.25, 0.25) node {3};
  \draw (3,1)--(4,0);
  \draw (3.75, 0.75) node {4};

\draw (4.5, 0.5) node {2};

\draw (2,1)--(2,2)--(3,2)--(3,1);
  \draw (2.5, 1.5) node {1};
\end{tikzpicture}\qquad
\begin{tikzpicture}[scale=0.6, font=\footnotesize, fill=black!20]
 \draw (1, 0) -- (4,0);
 \draw (1, 1) -- (4,1);
  \foreach \x in {1,2,3,4}
    {
      \draw (\x, 0) -- (\x,1);
    }
  \draw (1.5, 0.5) node {0};    
  \draw (2.5, 0.5) node {2};
  
  \draw (3.25, 0.25) node {3};
  \draw (3,1)--(4,0);
  \draw (3.75, 0.75) node {4};
\draw (2,1)--(2,2)--(3,2)--(3,1);
  \draw (2.5, 1.5) node {1};
\draw (3,2)--(4,2)--(4,1);
\draw (3.5, 1.5) node {2};
\end{tikzpicture}\qquad
\begin{tikzpicture}[scale=0.6, font=\footnotesize, fill=black!20]
 \draw (1, 0) -- (3,0);
 \draw (1, 1) -- (4,1);
  \foreach \x in {1,2,3}
    {
      \draw (\x, 0) -- (\x,1);
    }
  \draw (1.5, 0.5) node {0};    
  \draw (2.5, 0.5) node {2};
  \draw (3,1)--(4,0)--(4,1);
  \draw (3.75, 0.75) node {4};
\draw (2,1)--(2,2)--(3,2)--(3,1);
  \draw (2.5, 1.5) node {1};
\draw (3,2)--(4,2)--(4,1);
\draw (3.5, 1.5) node {2};

 \draw (4,2)--(5,1)--(5,2)--(4,2);
\draw (4.75, 1.75) node {3};
\end{tikzpicture}\qquad
\begin{tikzpicture}[scale=0.6, font=\footnotesize, fill=black!20]
 \draw (1, 0) -- (6,0);
  \foreach \x in {1,2,3}
    {
      \draw (\x, 0) -- (\x,\x);
    }
  
   \foreach \y in {1,2,3}
       {
            \draw (\y,\y) -- (6,\y);
       }
 \draw (4, 0) -- (4,3);
\draw (5, 0) -- (5,3);
 \draw (6, 0) -- (6,3);
 \draw (6, 3) -- (7,3) -- (7,2) --(6,2);

\foreach \x in {0,1,2}
        {
        	\draw (\x+2.5, \x+0.5) node {2};
        }
  \foreach \x in {0,1,2}
       {
           \draw (\x+4.5, \x+0.5) node {2};
          }   
     
     \foreach \x in {0}
         {
                	\draw (\x+1.5, \x+0.5) node {0};
                	\draw (\x+2.5, \x+1.5) node {1};
                }
       \draw (3.5, 2.5) node {0};
 \foreach \x in {0,2}
           {
           \draw (\x+3.75,\x+0.75) node {4};
                        	\draw (\x+3.25,\x+0.25) node {3};
                        	\draw (\x+3,\x+1) -- (\x+4,\x+0);
           }
  \foreach \x in {1}
             {
             \draw (\x+3.75,\x+0.75) node {3};
                                     	\draw (\x+3.25,\x+0.25) node {4};
                                     	\draw (\x+3,\x+1) -- (\x+4,\x+0);
             }
      \foreach \x in {0}
                    {
                    \draw (\x+5.75,\x+0.75) node {0};
                                 	\draw (\x+5.25,\x+0.25) node {1};
                                 	\draw (\x+6,\x) -- (\x+5,\x+1);
                    }
     \draw (6,2)--(7,1)--(7,2);
     \draw (6.75,1.75) node {1};
     \draw (7,3)--(8,2)--(8,3)--(7,3);
     \draw (7.75,2.75) node {0};
     
\end{tikzpicture}
\fi
\caption{Some $D_4$-partitions}
\label{fig:D4}
\end{figure}

Be that as it may, to a $D_r$-partition $\beta\in\PPP_{D_r}$ we can associate a set of $r+1$ nonnegative integers \[\w(\beta)=(w_0(\beta), \ldots, w_r(\beta))\]
counting the numbers of (full or half) boxes in $\beta$ of different colours (labels). For a set of variables $\xx=(t_0, \ldots, t_r)$,
we are lead to the generating function
\[ Z^{BD_{r}}(\xx) = \sum_{\beta\in\PPP_{D_r}}\xx^{\w(\beta)}.
\]
The following result is the precise analogue of Theorem~\ref{thm_abelian2d} above. 

\begin{theorem} Set $t=t_0t_1t_{r-1}t_r\prod_{j=2}^{r-2}t_j^2$. Then there is a factorisation
\begin{equation}Z^{BD_{r}}(\xx) = E(t)^{r+1} \cdot Z_\red^{BD_{r}}(\xx)
\label{2dfactorisation_Dn}
\end{equation}
into a power of the function $E(t)$ and a reduced series $Z_\red^{BD_{r}}\in\N[[\xx]]$. This latter 
series is itself a generating function
\begin{equation} Z_\red^{BD_{r}}(\xx) = \sum_{\lambda\in\PPP_{D_r}^{(r)}}\xx^{\w(\lambda)},
\end{equation}
where $\PPP_{D_r}^{(r)}\subset\PPP$ is a certain set of {\em core $D_r$-partitions}. Moreover, this function admits the following expression:
\[Z_\red^{BD_{r}}(\xx) = \sum_{\m\in\Lambda_{D_r}}t^{\frac{\langle \m, \m\rangle}{2}}\prod_{i=1}^{r} t_i^{m_i},
\]
where $\Lambda_{D_r}\cong (\Z^{r}, \langle\  \rangle)$ is the type $D_{r}$ root (Cartan) lattice, the lattice corresponding to the Dynkin diagram of type
$D_{r}$. 
\label{thm_Dnseries}
\end{theorem}
\begin{proof} The proof works along the same lines as the proof of Theorem~\ref{thm_abelian2d}; details are spelled out in~\cite{GyNSz}.  
\end{proof} 
We are not aware of a proof of the obvious generalization of formula~\eqref{eq:typeAsurfacerefined} for Poincar\'e polynomials in the type $D$
case. 

\begin{remark} \rm As already hinted above, the set of all $D_r$-partitions gives one combinatorial model
of the {crystal basis} of the basic representation of an infinite-dimensional Lie algebra closely related to the affine root system $\widehat D_{r}$; 
see~\cite{KK}.
\end{remark}

For subgroups of $\SL_2(\C)$ of types $E_6, E_7, E_8$, in other words the symmetry groups of the tetrahedron, cube and dodecahedron, 
we are not aware of such a construction of the crystal basis. Perhaps the corresponding groups are too ``three-dimensional''. 

\subsection{Groups outside $\SL$}

For any abelian group $G<\GL_2(\C)$, the colouring of partitions by periodic colours of course makes sense. However, we are not aware of 
any results computing the resulting partition function in full if $G$ contains elements of determinant different from $1$. See~\cite{BJ} 
for some nontrivial results. 

\section{Plane partitions: combinatorics}

\subsection{The setup and the main conjecture}

Now let $G$ be a finite abelian subgroup of $\SL_3(\C)$ of size $r$, that we will assume again to be diagonally embedded.
As before, let $\Irrep(G)=\{\rho_0={\rm triv},\ldots,\rho_{r-1}\}$ denote the irreducible representations of~$G$.  
The embedding  $G<\SL_3(\C)$ corresponds to a three-dimensional representation of $G$ that can be decomposed as  $\C^3\cong \rho_a\oplus\rho_b\oplus\rho_c$; if needed, we will denote this group $G(a,b,c)<\SL_3$. 

As before, we can colour (label) each box of a plane partition $\alpha\in\QQQ$ by one of $r$ labels as follows. 
Each box of coordinate $(i,j,k)$ in $\alpha$ corresponds to a monomial 
$x^iy^jz^k\in\C[\N^3]$, and thus corresponds to a one-dimensional $G$-eigenspace with character $\rho_{\chi(i,j,k)}\in\Irrep(G)$ for some $\chi(i,j,k)\in\{0,\ldots, r-1\}$; give this box colour $\chi(i,j,k)$. Let $w_i(\alpha)$ to be the number of boxes in $\alpha$ of colour $i$ and set $\w(\alpha)=(w_0(\alpha), \ldots, w_{r-1}(\alpha))$. See Figure~\ref{fig:3dcoloured} for some plane partitions coloured by $\mu_3(1,1,1)<\SL_3$.

\begin{figure}
\iflogvar
\centering
\begin{minipage}{.4\textwidth}
\begin{tikzpicture}[scale=0.5]
\colorplanepartition{{2,2,1},{1},{1},{1}}
\end{tikzpicture}
\end{minipage}
\begin{minipage}{.4\textwidth}
\begin{tikzpicture}[scale=0.5]
\colorplanepartition{{5,3,2,2,2,2,1},{5,3,2,2},{2,2,1,1},{2,2},{1},{1}}
\end{tikzpicture}
\end{minipage}
\fi
\caption{Plane partitions $\alpha_1, \alpha_2$ coloured by $\mu_3(1,1,1)<\SL_3$. The weights are $\w(\alpha_1)=(2,3,3)$, respectively $\w(\alpha_2)=(13,14,14)$}
\label{fig:3dcoloured}
\end{figure}
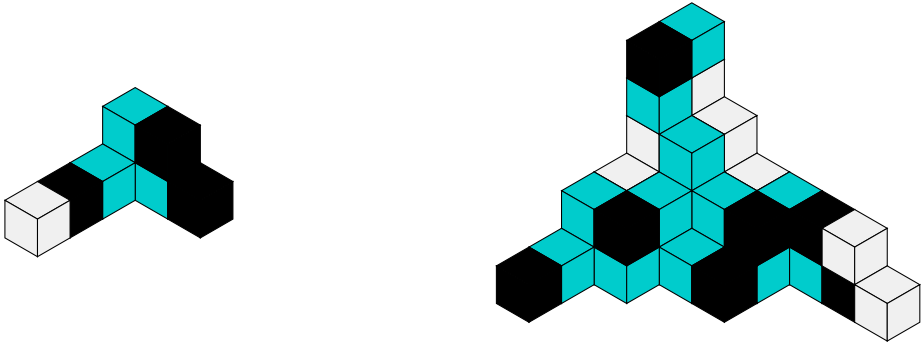

Let
\begin{equation}
\label{eq:3d:Zdef}
Z^{G(a,b,c)}(\xx)=\sum_{\alpha\in\QQQ}\xx^{\w(\alpha)}
\end{equation}
be the generating function of $G(a,b,c)$-coloured plane partitions. Once again, setting $t_i=t$ for all $i$ recovers the generating function $M(t)$ of all plane partitions. 

Motivated on the one hand by the formal analogy with coloured partitions, and on the other by some geometric ideas to be discussed in the last part of the paper, we make the following conjecture about this generating function. 

\begin{conjecture} \label{main_conj}Setting $t=\prod_{j=0}^{r-1}t_j$ to be the product of the variables, there exists a factorisation 
\begin{equation} Z^{G(a,b,c)}(\xx) = M(t)^r\cdot Z^{G(a,b,c)}_\red(\xx),
\label{formula_factor}
\end{equation}
where the reduced series $Z^{G(a,b,c)}_\red(\xx)$ has non-negative coefficients only.
\label{conj_comb}
\end{conjecture}

The series $M(t)^r$ begins $1+rt_0t_1\ldots t_{r-1}+O(t^{2})$. Thus our putative factorisation with positive quotient can only work if the coefficient
of $t_0t_1\ldots t_{r-1}$ in $Z^{G(a,b,c)}(\xx)$ is at least~$r$. In fact, the following stronger statement is implicit in the literature. 

\begin{proposition} The coefficient of the $t_0t_1\ldots t_{r-1}$ term in $Z^{G(a,b,c)}(\xx)$ is exactly $r$. In other words, 
independently of the choice of the group $G<\SL_3$, there are exactly $r$ plane partitions with one box of colour $i$ for each $i=0,\ldots, r-1$. 
\label{prop:equalsr}
\end{proposition}
\begin{proof} This is stated in~\cite[Lemma 4.1]{Nak}, but, it seems to us, without proof. In the next section, we discuss a geometric argument for this statement. We are not aware of a direct combinatorial proof. 
\end{proof}

It is natural to speculate whether the reduced series $Z^{G(a,b,c)}_\red(\xx)$ might be the generating function 
of some natural subset of the set  $\QQQ$ of plane partitions, analogous to the second part of Theorems~\ref{thm_abelian2d}, \ref{thm_Dnseries}. 
Unfortunately, as we discuss in one of the examples below in Section~\ref{sec_Z3}, this cannot be true, at least without breaking the natural symmetries of 
the problem. 

In the following sections we discuss some instances in which Conjecture~\ref{conj_comb} is known to hold, and present some numerical evidence for it in other cases. Section 4 of the paper puts the conjecture in a geometric context, discusses a quantization of it, and outlines a possible line of attack.  

\subsection{The action $\mu_r(1,r-1,0)$ and $\sl_{r}$ symmetry}
\label{3dAn}

Consider the case when $G<\SL_2(\C)<\SL_3(\C)$ is abelian; thus $G\cong \mu_r$ acts with weights $(1,r-1,0)$. 
The generating function $Z^{\mu_r(1,r-1,0)}(\xx)$ admits an explicit infinite product representation. 

\begin{theorem}[Young~\cite{BY}] We have 
\begin{equation}Z^{\mu_r(1,r-1,0)}(t_0\ldots, t_{r-1})=M(t)^r \prod_{0<a<b<r}\tilde M\left(\prod_{a\leq j\leq b} t_j, t\right).\label{eq:3dAn}\end{equation}
\end{theorem}
\begin{proof} As discussed in~\cite{BY}, some of the standard proofs of~\eqref{eq_mcmahon} also prove this result. The key is to ``slice'' a plane partition into partition slices that are monochromatic in this colouring. 
\end{proof}
In particular, Conjecture~\ref{conj_comb} holds in this case. To bring out symmetries of this formula more explicitly, we can re-write it using the plethystic exponential as follows. 
\begin{corollary} Let $\Sigma\subset \Lambda_{A_{r-1}} \cong \Z^{r-1}\subset \Z^{r}$ denote the root system of the Lie algebra of type $A_{r-1}$.  Then
\[Z^{\mu_r(1,r-1,0)}(\xx)= \PExp^\sigma\left( \frac{t}{(1-t)^2}\left(r+\sum_{\alpha\in\Sigma}\xx^{\alpha} \right)\right), \]
where $t=\prod_{j=0}^{r-1}t_j$ still, and in the last  bracket, only the variables $t_1, \ldots, t_{r-1}$ appear.
\end{corollary}
This formula nicely confirms expectations of the theoretical physics literature, where this geometry is expected to lead to ``enhanced $A_{r-1}$ symmetry''~\cite{EGS}. The $\sl_{r}$-symmetry of the last formula is apparent: it is essentially the (exponential of the) character of the adjoint representation. 

\subsection{The case $G\cong \mu_2\times\mu_2$ and $\sl_2^3$-symmetry}
\label{Z2Z2}

The main result of~\cite{BY} treats the case $G=\mu_2\times\mu_2<\SL_3(\C)$, embedded as the group of all diagonal matrices with 
determinant $1$ that square to the identity. The group $G$ has four one-dimensional irreducible representations
$\Irrep(G)=\{\rho_0={\rm triv},\rho_1, \rho_2, \rho_3\}$. The $G$-coloured partition function again admits an explicit infinite product representation.

\begin{theorem} [Young~\cite{BY}] With $t=\prod_{i=0}^3 t_i$ as before, 
\[Z^{\mu_2\times\mu_2}(t_0, t_1, t_2, t_3)= M(t)^4 \frac{\tilde M(t_1t_2, t)\tilde M(t_1t_2, t)\tilde M(t_1t_2, t)}{\tilde M(-t_1, t)\tilde M(-t_2, t)\tilde M(-t_3, t)\tilde M(-t_1t_2t_3, t)}.\]
\end{theorem}
\begin{proof} This is much harder; the proof in~\cite{BY}  is indirect, and uses combinatorial wall-crossing to arrive at a different problem, solved
earlier, involving pyramid partitions. 
\end{proof}
This result confirms Conjecture~\ref{conj_comb} in this case. 
The right hand side can again be written in a more symmetric form using the plethystic exponential. Indeed, a short calculation shows

\begin{corollary} We have 
\[Z^{\mu_2\times\mu_2}(\xx)= \PExp^\sigma\left( \frac{t}{(1-t)^2}\left(2+(t_1^{\half} -t_1^{-\half})(t_2^{\half} -t_2^{-\half})(t_3^{\half} -t_3^{-\half}) ((t_1t_2t_3)^{-\half} -(t_1t_2t_3)^{\half}) \right)\right) \]
\end{corollary}

This formula  is compatible with expectations in the physics that expects $\sl_2^3$-symmetry in the partition functions connected to this geometry~\cite{BBT}. Indeed, the first three terms in the last bracket give the character of the fundamental representation of $\sl_2^3$. Introducing a new variable $t_4$ with $t_1t_2t_3t_4=1$, we can re-write the last formula as
\[Z^{\mu_2\times\mu_2}(\xx)= \PExp^\sigma\left( \frac{t}{(1-t)^2}\left(2+\prod_{j=1}^4(t_j^{\half} -t_j^{-\half}) \right)\right), \]
suggesting that there is more symmetry around. It is not clear to us whether there is any significance to this remark. 

\subsection{The action $\mu_3(1,1,1)$ and other cyclic cases}
\label{sec_Z3}

The first case not covered by the results above is when $G\cong\mu_3$, acting with weights $(1,1,1)$ on $\C^3$. 
The generating function $Z^{\mu_3(1,1,1)}(\xx)$ does not appear to admit an infinite product expansion. 
We checked Conjecture~\ref{conj_comb} with the
help of computer enumeration up to degree $24$; we include the resulting reduced series $Z^{\mu_3(1,1,1)}_\red(\xx)$ in the Appendix. The support of this series, the rational cone in ${\bf N}^3$ spanned by $(a_0,a_1,a_2)$ with the corresponding monomial $t_0^{a_0}t_1^{a_1}t_2^{a_2}$ having nonzero coefficient, appears to be the simplicial cone generated by $(1,0,0)$, $(1,3,0)$ and $(1,3,6)$. These monomials come with coefficient $1$, and correspond to the first, second and third power of the maximal ideal of the origin respectively.



As we mentioned before, it would have been natural to hope that the reduced series $Z^{G(a,b,c)}_\red(\xx)$ is the coloured generating function of some natural subset of the set  
$\QQQ$ of plane partitions.
However, in the case $G=\mu_3(1,1,1)$, we show that there is no such subset chosen in a way that fully respects the symmetries of 
the problem. Let $S_3$ be the natural group of (rotation and reflection) symmetries of the positive octant $\N^3$. 
Note that the colouring employed in this case is $S_3$-invariant. 

\begin{proposition} \label{nogo_prop}For the action of the group $\mu_3$ with weights $(1,1,1)$ on $\C^3$, there does not exist an $S_3$-invariant rule that selects a 
subset of $\QQQ$ whose coloured generating function agrees with the reduced series
$Z^{\mu_3(1,1,1)}_\red(\xx)$.\label{no_count}
\end{proposition} 
\begin{proof} It can be checked by computer enumeration that the number of plane partitions of multi-weight 
$(3,3,3)$ is 108. Furthermore, this set is partitioned into orbits of size $3$ or $6$ under the natural action of $S_3$: there are
no plane partitions of this multi-weight with full $S_3$ or $A_3$-symmetry. 
However, the corresponding coefficient
of the reduced series $Z^{\mu_3(1,1,1)}_\red(\xx)$ is $44$. This number is not divisible by $3$, so there is no symmetric
rule that would select 44 partitions among the 108 with this multi-weight.
\end{proof}


We further checked Conjecture~\ref{conj_comb} to degree 24 by computer enumeration also for the cases $\mu_4(1,1,2)$, 
$\mu_5(1,1,3)$ and $\mu_6(1,2,3)$.

\subsection{The case $G\cong \mu_3\times\mu_3$ and the elusive ${\mathfrak e}_6$-symmetry}\label{sec:e6} We discuss one last case; consider 
$G\cong \mu_3\times\mu_3$ embedded in $\SL_3(\C)$. Numerical evidence is of limited value for this example, since Proposition~\ref{prop:equalsr} means that the non-trivial checks for Conjecture~\ref{conj_comb} happen in total degree $18$ and above, and it is increasingly difficult to calculate values in this range; but the conjecture does pass the tests we were able to set it (up to degree 20). On the other hand, physics not only predicts the existence of a $\sl_3^3$-symmetry in this example, in analogy with Section~\ref{Z2Z2}, but an enhancement of this symmetry to ${\mathfrak e}_6$; see~\cite[Section 4.2]{BBT}. Indeed, one of the resolutions of the quotient $\C^3/G$ is the total space of the anticanonical
bundle on a specialised cubic surface, and thus symmetry under ${\mathfrak e}_6$ is expected. It would be interesting to see some remnant of
this symmetry in the reduced orbifold series $Z^{\mu_3\times\mu_3}(\xx)$, but we have not been able to do so. 

\section{Plane partitions: geometry}

\subsection{The geometry of the quotient and its resolutions}

The action of $G<\SL_3(\C)$ on $\C^3$ preserves the three-form $dx\wedge dy\wedge dz$, and hence the quotient $X=\C^3/G$ has Gorenstein (Calabi--Yau) singularities; the singular set of $X$ includes the origin, and can also contain 
curves of singularities. In the three-dimensional case, there is still a Calabi--Yau resolution, but it is not unique. There turns out to be a distinguished
resolution~$Y_G$, Nakamura's $G$-Hilbert scheme~\cite{Nak}; the resolution map is given by the Hilbert--Chow morphism $\pi_G\colon Y_G \rightarrow X=\CC^3/G$. Other resolutions $\pi_i\colon Y_i\to X$
are obtained by flops from~$Y_G$. One key feature \cite[Theorem 1.10]{Bat} of the McKay correspondence
is the numerical equality of the topological Euler characteristic of one (hence all) Calabi--Yau resolution(s) of $X$, and the ``orbifold Euler characteristic'' of the equivariant geometry, which in the case of $G<\SL_3(\C)$ acting on $\C^3$ is just the number of conjugacy classes of $G$, or for $G$ abelian, simply its order $r$.

In simple cases, all fibres of the resolution $\pi_G\colon Y_G\to X$ are at most one-dimensional; the examples discussed in Sections~\ref{3dAn}-\ref{Z2Z2} were of this type. In general however, the map $\pi_G$ has a two-dimensional fibre over the (image of the) origin $0\in X$, leading to increased complexity in the problem, as shown by the examples in Sections~\ref{sec_Z3}-\ref{sec:e6}. Indeed, for the example $G\cong\mu_3$ acting with weights $(1,1,1)$ on $\C^3$, the quotient $X=\C^3/G$ admits
the (unique) Calabi--Yau resolution $Y=Y_G={\mathcal O}_{\P^2}(-3)$ with $\pi\colon Y\to X$ contracting the zero section. Thus the 
resolution of singularities $\pi$ of $X$ has the two dimensional fibre $\P^2$ over the origin  (smooth and irreducible in this case). 

The tautological bundle of $\CC[x,y,z]\rtimes G$-modules on $\ghilb{G}$, considered as a $G$-equivariant sheaf on $\ghilb{G}\times \CC[x,y,z]$, provides the Fourier--Mukai kernel of an equivalence of categories~\cite{BKR} 
\begin{equation}\label{eq:FM}
\Phi\colon \Db{\Coh(\ghilb{G})}\rightarrow \Db{(\CC[x,y,z]\rtimes G)\lmod}
\end{equation}
with an inverse that we will denote by $\Psi$.  

\subsection{Plane partitions and Euler characteristics}

In complete analogy with the discussion of Section~\ref{sec:2dHilb}, the generating function $Z^{G(a,b,c)}(\xx)$ can be thought of as the generating function of topological Euler characteristics of spaces of $G$-invariant subschemes of $\C^3$, or equivalently sheaves on the orbifold $[\C^3/G]$.
The equivariant Hilbert schemes $\Hilb^{\rho}([\CC^3/G])$ for $\rho\in\Rep(G)$ are in general singular, but still have finitely many fixed 
points under the action of the group $T$ of diagonal matrices in $\GL_3(\CC)$, parametrised by plane partitions (we are once again relying on the fact here that $G$ is abelian so can be assumed to be diagonal). We thus get an equality
\begin{equation}
Z^{G(a,b,c)}(\xx)=\sum_{\dd\in\mathbb{N}^r} \chi(\Hilb^{\rho_{\dd}}([\CC^3/G]))\xx^{\dd},
\label{eq:3d:Z}
\end{equation}
with $\rho_{\dd}=\bigoplus_{0\leq i \leq r-1}\rho_i^{\oplus d_{i}}$.

Nakamura's $G$-Hilbert scheme is by definition $Y_G=\Hilb^{\rho_{\rm reg}}([\CC^3/G])$, where ${\rho_{\rm reg}}\in\Rep(G)$ is the regular
representation of $G$. This particular Hilbert scheme is known to be nonsingular~\cite{Nak, BKR} and indeed is a crepant resolution of the 
quotient $X=\C^3/G$, as noted already above. This fact then provides a proof of~Proposition~\ref{prop:equalsr}.

\begin{proof}[Proof of Proposition~\ref{prop:equalsr}] By the previous discussion, remembering that $G$ is abelian, the coefficient of
$t=\prod_{i=0}^{r-1}t_i$ in $Z^{G(a,b,c)}(\xx)$ is equal to the Euler characteristic $\chi(Y_G)$ of Nakamura's $G$-Hilbert scheme. Again as commented
above, this number is known to be equal to the ``orbifold Euler characteristic'' of $[\C^3/G]$ which is simply $r$ since $G$ is abelian~\cite{Bat}. 
\end{proof}

In the three-dimensional case, alongside the series $Z^{G(a,b,c)}(\xx)$ defined in~\eqref{eq:3d:Zdef} and identified with a partition function of Euler characteristics in (\ref{eq:3d:Z}), it is also natural to consider a {\em signed} series
\begin{equation}
Z_{\rm signed}^{G(a,b,c)}(\xx)= \sum_{\alpha\in\QQQ}(-1)^{d_0+(\dd,\dd)}\xx^{\w(\alpha)}
\end{equation}
where the bilinear form in the exponent is defined by
\[
(\dd,\ee)=\sum_{i \textrm{ a vertex of }Q} d_i e_i-\sum_{a:i\rightarrow j\textrm{ an arrow of }Q} d_i e_j.
\]
and the sign change will arise from taking a virtual Euler characteristic, ubiquitous in three-dimensional sheaf counting problems. In the abelian case, this is still 
a sum over torus-fixed points~\cite{BF}, so the generating function gives a signed count of planar partitions~\cite[A.2]{BY}. We will return to these signs below.

\subsection{Representations of the McKay quiver in dimension $3$}

We can also approach coloured plane partitions via moduli spaces of representations of a McKay quiver, just as in dimension 2.  We will denote by $Q=Q{(a,b,c)}$ the McKay quiver for $G(a,b,c)$, constructed as follows. The vertex set $Q_0$ of $Q=Q{(a,b,c)}$ is still identified with the set of irreducible representations $\rho_0,\ldots,\rho_{r-1}$ of $G$.  For each $i\leq r-1$ there is an arrow $x_i$ going from $\rho_i$ to $\rho_{i+a}$, where addition is modulo $r$, and similarly, arrows $y_i,z_i$ from $\rho_i$ to $\rho_{i+b}$ and $\rho_{i+c}$ respectively.  As before, $G$-equivariant sheaves on $\CC[x,y,z]$ are naturally isomorphic to representations of this quiver, satisfying the obvious quadratic (commutation) relations.  A new feature in the 3-dimensional setting is that these relations can be packaged as the noncommutative derivatives~\cite{Gin} of a potential $W=W{(a,b,c)}$:
\begin{equation}
\label{MP:eq}
W{(a,b,c)}=\sum_{0\leq i\leq r-1}\left( z_{i+a+b}y_{i+a}x_i-y_{i+a+c}z_{i+a}x_i \right).
\end{equation}
The variables should be considered as (non-commuting) quiver arrows here, and subscripts should still be taken modulo $r$.  Since $a+b+c=0 \ \mod \ r$, the above expression is a sum of cyclic paths in $Q$.  We write $\mathbb{C}(Q,W)$ for the quotient of the free path algebra by the noncommutative derivatives of $W$. Then there is an isomorphism
\[
\CC[x,y,z]\rtimes G\cong \CC(Q,W)
\]
which at the level of underlying vector spaces of representations, takes a $\CC[x,y,z]\rtimes G$-module to its decomposition according to the characters of $G$. 

It was observed by Ito and Nakajima \cite[Section 3]{NI}, that Nakamura's G-Hilbert scheme $Y_G$ makes an appearance as a moduli space of stable representations of the McKay quiver with potential $(Q,W)$.  Recall that a stability parameter~\cite{Kin} for a quiver $Q$ is an element
\[\zeta\in\Q^{Q_0}.\]
For concreteness, for the rest of this paper we fix the stability parameter
\begin{equation}\label{def:stab}\zeta_0=(r-1,-1,\ldots,-1)\in \Q^{Q_0} \end{equation}
on the McKay quiver $Q$. The slope of a $Q$-representation $\rho$ with respect to 
a stability parameter $\zeta$ is defined by
\[
\mu(\rho)=\frac{\dimvect(\rho)\cdot\zeta}{\dim(\rho)}
\]
where $\dimvect(\rho)\in\mathbb{N}^{Q_0}$ is the dimension vector of the representation, and $\dim(\rho)=\sum \dimvect(\rho)_i$ is the ordinary dimension.  A representation $M$ is defined to be $\zeta$-(semi)stable if all proper submodules have slope less than (or equal to) that of $M$.  Then 
we have

\begin{proposition} {\rm (Ito and Nakajima)} Fix the dimension vector $\dd_0=(1,\ldots,1)$. There exists an open neighbourhood~$U$ of stability parameters of $\zeta_0\in \Q^{Q_0}$, such that the space $Y_G$ can be identified with the fine moduli space of $\zeta$-stable $\dd_0$-dimensional $\CC(Q,W)$-modules for any $\zeta\in U$.
\label{prop:NI}
\end{proposition}   

More generally, there is an isomorphism of schemes between $\Hilb^{\rho_{\dd}}([\CC/G])$ and the moduli space $\ncHilb^{\dd}(Q,W)$ of pairs of a $\dd$-dimensional $\CC(Q,W)$-representation $M$ with a generating vector $v\in M^G$, which in turn is realised as a subscheme of the smooth moduli scheme of all stable framed representations of the free path algebra $\CC Q$ as the critical locus of the function $\Tr(W)$. 
To study such representations, we will write $Q'$ for the quiver obtained from the McKay quiver $Q$ by adding one extra vertex, which we label $\infty$, and one arrow, from $\infty$ to $\rho_0$.  

Let $T_0\cong(\CC^*)^2\subset T\cong(\CC^*)^3$ be the sub-torus of the torus $T$ acting on $\CC^3$ and fixing the three-form $dx\wedge dy\wedge dz$.  The function $\Tr(W)$ is $T_0$-invariant, and its fixed locus has isolated fixed points equal to the fixed points of the $T$-action on $\ncHilb^{\dd}(Q,W)$; the latter are precisely parametrised by coloured plane partitions.  The results of~\cite{BF} then tell us that the correct weighted Euler characteristic to consider is given by summing over these fixed points, with each fixed point $p$ contributing $(-1)^{T_p Y}$, the parity of the Zariski tangent space at $p$, which is equal to $(-1)^{d_0+(\dd,\dd)}$ by \cite[Theorem 7.1]{Til}.
In conclusion, there is an equality
\begin{equation}
\label{BF_eq}
Z_{\rm signed}^{G(a,b,c)}(\xx)=\sum_{\dd\in\mathbb{N}^r}\chi_{\rm vir}(\ncHilb^\dd(Q,W))\xx^\dd
\end{equation}
between the signed version of the coloured plane partition generating function, and the so-called Donaldson--Thomas (DT) generating
function attached to the collection of spaces $\{\ncHilb^\dd(Q,W): {\dd\in\mathbb{N}^r}\}$.

\begin{remark} \rm In~\cite{MNOP}, Maulik et al.~attach an analogous generating
function to ordinary (non-orbifold) Calabi--Yau geometries\footnote{The conjecture in~\cite{MNOP} is more general; we only need the Calabi--Yau case.} $Y$ with topological Euler characteristic $e(Y)$, and conjecture a factorisation into a term $M(-t)^{e(Y)}$, corresponding to point sheaves on $Y$, and a reduced series.  
One way to think about Conjecture~\ref{conj_comb} is as an orbifold analogue of the conjecture of~\cite{MNOP}, where $r=e(Y)=e_{\rm orb}(\C^3, G)$ is the (orbifold) Euler characteristic.
\label{rem:mnop}
\end{remark}

\subsection{Refinements from cohomological DT theory}\label{refin_sec}
The link between DT theory and coloured plane partitions can be enriched by considering mixed Hodge structures on the cohomology of the moduli spaces $\ncHilb^{\dd}(Q,W)$ instead of their weighted Euler characteristics. The technical tools for refined Donaldson--Thomas theory come from Saito's theory of mixed Hodge modules.

Given a function $f$ on a smooth variety $Z$, we write\footnote{Strictly speaking, to make sense of this expression when the exponent is not an integer, we have to formally add a tensor square root $\LL^{1/2}$ to the category of mixed Hodge modules --- see \cite{KS2} for details.}
\[
\phi_f:=\varphi_f\underline{\mathbb{Q}}_Z\otimes \LL^{-\dim(Z)/2}
\]
for the normalised mixed Hodge module of vanishing cycles\footnote{In the general framework of cohomological Donaldson--Thomas theory, the cohomology of such an object is supposed to be considered as a \textit{monodromic} mixed Hodge structure; however in our case, since the quiver with potential admits a cut, in the sense of \cite{HeIy}, the monodromy is trivial, see the appendix of \cite{Dav}.}, where $\LL=\HO_c(\mathbb{A}^1,\mathbb{Q})$ is a pure weight $2$ Hodge structure in cohomological degree $2$. We package the mixed Hodge structures corresponding to all $\dd\in\mathbb{N}^{r}$ together into a $\mathbb{N}^{r}$-graded mixed Hodge structure
\[
\HHhilb(a,b,c):=\bigoplus_{\dd\in\mathbb{N}^r}\HO(\ncHilb^\dd(Q),\phi_{\Tr(W)}).
\]
On the right hand side we have taken the cohomology of the smooth ambient moduli scheme $\ncHilb^{\dd}(Q)$; since the mixed Hodge module $\phi_{\Tr(W)}$ is supported on $\ncHilb^{\dd}(Q,W)$, the critical locus of $\Tr(W)$, this makes no difference. In terms of stability conditions, $\ncHilb^{\dd}(Q)$ can be realised as the (fine) moduli space of $\zeta'$-stable $(\dd,1)$-dimensional $Q'$-representations, for $\zeta'= (-1,\ldots,-1,\lvert \dd\lvert)$.  

For later use, we generalise this construction as in \cite[Section 3]{EnRe} to take account of nontrivial stability conditions on $Q$.  Let $\zeta\in\mathbb{N}_{Q_0}$ be such a stability condition.  Fixing a dimension vector $\dd$ and $M\gg 0$ a large integer, we may assume that $\zeta\cdot \dd=-1/M$ by adding a constant dimension vector to $\zeta$, leaving the notion of $\zeta$-(semi)stability unchanged.  Then we set $\zeta'=(M\zeta,1)\in\mathbb{Q}^{Q'_0}$.  It is easy to check that a $(\dd,1)$-dimensional $Q'$-representation is stable if it is semistable, and it is semistable precisely if the underlying $Q$-representation $\rho$ is $\zeta$-semistable and there is no proper subrepresentation $\rho'\subset \rho$ containing the image of the one-dimensional vector space at the vertex $\infty$, and such that $\mu(\rho')=\mu(\rho)$.  We denote the resulting fine moduli space $\ncHilb^{\zeta,\dd}(Q)$.  For a stability parameter $\zeta$, slope function $\mu$ and fixed slope $\theta\in (-\infty,\infty)$ we define,
\[
\HHhilb(a,b,c)_\theta:=\bigoplus_{\dd\in\mathbb{N}^{Q_0}\lvert \slope(\dd)=\theta}\HO(\ncHilb^{\zeta,\dd}(Q),\phi_{\Tr(W)}).
\]

\begin{definition}\label{cf_def}
Let $\mathcal{L}$ be a $\mathbb{N}^{Q_0}$-graded, cohomologically graded mixed Hodge structure, such that for each $\dd\in\mathbb{N}^r$ there is an equality $\gw{n}\HO(\mathcal{L}_{\dd})=0$ for $n\gg0$, and moreover for every $n$, $\gw{n}\HO^i(\mathcal{L}_{\dd})=0$ for all but finitely many $i$.  We define the characteristic function of $\mathcal{L}$:
\begin{equation}
\label{cf_exp}
\chi_{Q,\wt}(\mathcal{L},\qq^{1/2},\xx)=\sum_{n,i\in\mathbb{Z}}\sum_{\dd\in\mathbb{N}^{Q_0}}(-1)^i\dim(\gw{n}\HO^i(\mathcal{L}_{\dd}))(\qq^{1/2})^n\xx^{\dd},
\end{equation}
as an element of the \textit{quantum torus} 
\[
\QT{Q}=\mathbb{Z}((\qq^{-1/2}))[[t_i\lvert_{i\in Q_0}]]
\]
endowed with the usual additive structure, and the $\mathbb{Z}((\qq^{-1/2}))$-linear multiplication defined on $t$-variables by $\xx^{\dd}\cdot \xx^{\dd'}=\qq^{(\dd,\dd')/2}\xx^{\dd+\dd'}$.

If the total cohomology of each of the graded pieces $\mathcal{L}_{\dd}$ is finite-dimensional, then we define
\[\chi_{Q}(\mathcal{L},\xx)=\chi_{Q,\wt}(\mathcal{L},q^{1/2}=1,\xx)\in\mathbb{Z}[[t_i\lvert_{i\in Q_0}]].\]  
\end{definition}

The mixed Hodge structure $\HHhilb(a,b,c)$ satisfies the stronger finite-dimensionality condition of Definition \ref{cf_def}, and by \cite{Beh} we have that the Euler characteristic of its $\dd$th graded piece 
$\chi(\HHhilb(a,b,c)_{\dd})$ is the virtual Euler characteristic of $\ncHilb^{\dd}(Q,W)$.  Combining this with (\ref{BF_eq}) we deduce that there is an equality
\begin{equation}
\label{ECV}
\chi_{Q}(\HHhilb(a,b,c),\xx)=\sum_{\alpha\in\QQQ}(-1)^{d_0+(\dd,\dd)}\xx^{\w(\alpha)} = Z_{\rm signed}^{G(a,b,c)}(\xx)
\end{equation}
giving an interpretation of the signed plane partition function in terms of the representation theory of the McKay quiver with potential.

We also consider the mixed Hodge structure on the vanishing cycle cohomology
\[
\HHst(a,b,c)=\bigoplus_{\dd\in\mathbb{N}^{Q_0}}\HO_c\left(\Mst_{\dd}(Q),\phi_{\Tr(W)}\right)
\]
 of the stack $\Mst_{\dd}(Q)$ of \textit{all} finite-dimensional $\CC Q$-modules.  The cohomological degree is unbounded below for every nonzero $\mathbb{N}^{Q_0}$-degree, and so only the $q$-refined characteristic function of Definition \ref{cf_def} is defined for this mixed Hodge structure. For a stability parameter $\zeta\in \Q^{Q_0}$ and slope $\theta\in (-\infty,\infty)$ we can also define
 \[
\HHst(a,b,c)_{\theta}=\bigoplus_{\dd\in\mathbb{N}^{Q_0}\lvert \slope(\dd)=\theta} \HO_c\left(\Mst^{\zeta-\sst}_{\dd}(Q),\phi_{\Tr(W)}\right)
\] 
where $\Mst^{\zeta-\sst}_{\dd}(Q)$ is the stack of $\zeta$-semistable modules with dimension vector $\dd$, and $\mu$ is the slope function defined with respect to $\zeta$.

All of these mixed Hodge structures are related to each other in an elegant way by wall crossing isomorphisms which are special cases of \cite[Theorem B]{DaMe}.  The isomorphisms we will need are
\begin{equation}\HHst(a,b,c)\cong\bigotimes_{-\infty\xrightarrow{\theta} \infty}^{\tw} \HHst(a,b,c)_{\theta}\label{WCI1}\end{equation}
\begin{equation}\HHhilb(a,b,c)\otimes^{\tw}\HHst(a,b,c)\cong \HHst(a,b,c)\otimes^{\tw}\mathbb{Q}_{1_{\infty}}\label{WCI2}\end{equation}
and
\begin{equation}\HHhilb(a,b,c)_{\theta}\otimes^{\tw}\HHst(a,b,c)_{\theta}\cong \HHst(a,b,c)_{\theta}\otimes^{\tw}\mathbb{Q}_{1_{\infty}}\label{WCI3}.\end{equation}
The isomorphism (\ref{WCI1}) takes place in the category of $\mathbb{N}^{r}$-graded mixed Hodge structures, while (\ref{WCI2}) and (\ref{WCI3}) take place in the category of $\mathbb{N}^{r+1}$-graded mixed Hodge structures, i.e.~the categories of mixed Hodge structures graded by dimension vectors for $Q$ and $Q'$ respectively. The category of $\mathbb{N}^r$-graded mixed Hodge structures is equipped with a twisted tensor product via the rule $\mathcal{N}\otimes^{\tw}\mathcal{P}=\mathcal{N}\otimes\mathcal{P}\otimes\LL^{\frac{(\ee,\dd)-(\dd,\ee)}{2}}$ for $\mathcal{N}$ and $\mathcal{P}$ homogeneous of degree $\dd$ and $\ee$ respectively.  We extend this to a tensor product for $\mathbb{N}^{r+1}$-graded mixed Hodge structures by using the Euler pairing for $Q'$ instead of $Q$.  Finally, the object $\mathbb{Q}_{1_{\infty}}$ above is the constant pure weight zero Hodge structure $\mathbb{Q}$, situated in $\mathbb{N}^{r+1}$-degree $(0,\ldots,0,1)$, so 
\[
\chi_{Q',\wt}(\mathbb{Q}_{1_{\infty}}, \qq^{1/2},\xx)=t_{\infty}.
\]

This twist is introduced so that the equation $\chi_{Q}(\mathcal{L}\otimes^{\tw}\mathcal{B})=\chi_Q(\mathcal{L})\chi_Q(\mathcal{B})$ holds, for $\mathcal{L},\mathcal{B}$ two $\mathbb{N}^r$-graded mixed Hodge structures, and the same statement holds for $\mathbb{N}^{r+1}$-graded mixed Hodge structures with $\chi_Q$ replaced by $\chi_{Q'}$.

Putting together (\ref{WCI2}) and (\ref{ECV}) we deduce that
\[
\chi_Q(\HHhilb(a,b,c),\xx)=\left[\Ad_{\chi_{Q,\wt}(\HHst(a,b,c), q^{1/2},\xx)}(t_{\infty})\right]_{\qq^{1/2}=1=t_{\infty}} = Z_{\rm signed}^{G(a,b,c)}(\xx),
\]
where, for the left hand side, we have given $\HHhilb(a,b,c)$ the $\mathbb{N}^{r}$ grading given by only remembering the dimension vector restricted to $Q$.  Motivated by the above discussion, we make the following
\begin{definition} The $q$-deformed coloured plane partition function is 
\[
Z^{G(a,b,c)}(\xx,\qq^{1/2})=\chi_{Q,\wt}(\HHhilb(a,b,c),\qq^{1/2},\xx)\in\QT{Q}.
\]
\end{definition}
\noindent Setting $\qq^{1/2}=1$ in $Z^{G(a,b,c)}(\xx,\qq^{1/2})$ recovers the signed partition function $Z_{\rm signed}^{G(a,b,c)}(\xx)$.

The connection between these ideas and our Conjecture \ref{main_conj} comes from the following 
construction. Recall the dimension vector $\dd_0=(1,\ldots,1)\in \N^{Q_0}$, stability parameter $\zeta_0=(r-1,-1,\ldots,-1)\in \Q^{Q_0}$, and the open neighbourhood $\zeta_0\in U\subset\Q^{Q_0}$ from Proposition~\ref{prop:NI}. Fixing a bound $N\in\N$, it is clearly possible to choose a small rational perturbation $\zeta_N\in U\subset\Q^{Q_0}$ such that a dimension vector $\dd$ with all components smaller than $N$ has fixed slope 
\[
\theta_0=\dd_0\cdot \zeta_0=0
\]
if and only if $\dd$ is a rational multiple of $\dd_0$. The mixed Hodge structure $\HHhilb(a,b,c)_{\theta_0}$ stabilises in each degree as we take the limit in the sequence of stability conditions $\zeta_0,\zeta_1,\ldots$ so the entire graded mixed Hodge structures limits to a ``generic'' mixed Hodge structure which we denote $\HHpts(a,b,c)$.  Likewise, we denote by $\HHstpts(a,b,c)$ the limit of the mixed Hodge structures $\HHst(a,b,c)_{\theta_0}$ as we vary the stability condition.

There is a natural inclusion of quantum tori $\mathcal{A}_Q\subset \mathcal{A}_{Q'}$.  Note that setting as usual $t=\prod_{i=0}^{r-1}t_i = \xx^{\bf 1}$, the ring $\mathbb{Z}((q^{-1/2}))[[t]]$ lies in the centre of $\QT{Q}$, but not the centre of $\QT{Q'}$.  Indeed, if $\alpha\in t_{\infty}\cdot \mathcal{A}_Q$ then applying $\chi_{Q',\wt}$ to (\ref{WCI3}), with $\theta=\mu(\bf 1)=0$, we deduce
\begin{equation}
\label{eq:pts}
\Ad_{\chi_{Q,\wt}(\HHstpts(a,b,c))}(\alpha)=\alpha\ast\chi_{Q',\wt}(\HHpts(a,b,c))
\end{equation}
and we have

\begin{proposition} With $t=\prod_{i=0}^{r-1} t_i$ as before, there is an equality
\begin{equation}\label{BF_import}
\chi_{Q}(\HHpts(a,b,c), \xx)=M(-t)^{\chi(\ghilb{G})}.
\end{equation}
\end{proposition}
\begin{proof}
Set $\theta=0$.  For the stability condition $\zeta_N$ constructed above, the geometric Proposition \ref{ray_prop} to be discussed below implies that $\HHhilb(a,b,c)_{\theta}$ is the vanishing cycle cohomology of the Hilbert scheme of points of $\ghilb{G}$ up to degree $(N,\ldots,N)$. Taking $N\to\infty$ and using the results of~\cite{BF}, the stated equality follows.
\end{proof}

This proposition, along with (\ref{eq:pts}), says that the operator $\Ad_{\chi_{Q,\wt}(\HHstpts(a,b,c))}$ is a q-refinement of multiplication by the right hand side of (\ref{BF_import}), the term we factor out in the definition of the reduced partition function in~\eqref{formula_factor}.  On the other hand, by (\ref{WCI2}) we can write 
\[
Z_{\rm signed}^{G(a,b,c)}(\xx,\qq^{1/2})=\left[\Ad_{\chi_{Q,\wt}(\HHst(a,b,c), q^{1/2},\xx)}(t_{\infty})\right]_{t_{\infty}=1}
\]
and by centrality of $\chi_{Q,\wt}(\HHstpts(a,b,c))$ in $\mathcal{A}_Q$ we can factorise
\[
\Ad_{\chi_{Q,\wt}(\HHst(a,b,c), q^{1/2},\xx)}=\Ad_{\chi_{Q,\wt}(\HHstpts(a,b,c),\qq^{1/2},\xx)}\circ \Ad_{\chi_{Q,\wt}(\HHst(a,b,c)', q^{1/2},\xx)}
\]
where
\[
\HHst(a,b,c)'=\bigotimes_{-\infty\xrightarrow{\theta\neq 0} \infty}^{\tw} \HHst(a,b,c)_{\theta}.
\]

The discussion above motivates the following definition.
\begin{definition}
We define the $\qq$-deformed reduced plane partition function
\begin{equation}
\label{red_p_fn}
Z_{\red}^{G(a,b,c)}(\xx,\qq^{1/2})=\chi_{Q,\wt}(\HHhilb(a,b,c), q^{1/2}, \xx)/\chi_{Q,\wt}(\HHpts(a,b,c), q^{1/2}, \xx).
\end{equation}
\end{definition}

We formulate the following, quantized version of Conjecture \ref{main_conj}:

\begin{conjecture}\label{qconj}
The $\qq$-deformed reduced partition function (\ref{red_p_fn}) can be written
\[
Z_{\red}^{G(a,b,c)}(\xx,\qq^{1/2})=\sum_{\dd\in\mathbb{N}^r}h_{\dd}(\qq)(-\qq^{1/2})^{d_0+(\dd,\dd)}\xx^{\dd},
\]
where $h_{\dd}(\qq)\in \mathbb{N}[q^{\pm 1}]$.
\end{conjecture}
Setting $\qq^{1/2}=1$ in Conjecture \ref{qconj} recovers Conjecture \ref{main_conj}.

\begin{example} \rm For the case $G=\mu_r(1,r-1,0)$ studied in Section~\ref{3dAn}, Conjecture~\ref{qconj} is also known to hold. In fact, 
\cite[Theorem 1.2]{Moz} gives a somewhat involved but explicit formula for the quantized generating function,
a quantum version of~\eqref{eq:3dAn}, from which Conjecture~\ref{qconj} can be derived easily.
\end{example}

\subsection{Purity and positivity}

As noted in Proposition \ref{nogo_prop}, in general there is little hope of proving positivity of the reduced partition function by showing that the $\dd$th coefficient counts a certain subset of weight $\dd$ coloured partitions.  Rather than proving positivity combinatorially, i.e.~by showing that the coefficients count elements in a set, we may try to prove positivity via categorification, i.e.~by showing that these numbers are the dimensions of vector spaces.  In this section we establish the base camp for such an attempt, and prove Theorem~\ref{thm:positivity}, the positivity of the \textit{unreduced} $q$-refined partition function for certain abelian $G(a,b,c)<\SL_3(\CC)$, by exactly this kind of approach.

Our general approach to proving a conjecture like Conjecture \ref{qconj} regarding a partition function $\mathcal{Z}(\qq^{1/2},\xx)\in\mathbb{Z}((\qq^{-1/2}))[[\xx]]$ is as follows:
\begin{enumerate}
\item
Show that $\mathcal{Z}(\qq^{1/2},\xx)=\chi_{\wt}(\mathcal{H}, q^{1/2}, \xx)$ for $\mathcal{H}$ some $\mathbb{N}^{r}$-graded, cohomologically graded mixed Hodge structure (for example, arising in Donaldson--Thomas theory).
\item
Show that $\mathcal{H}$ is pure, in the sense that the $i$th cohomologically graded piece of $\mathcal{H}$ is pure of weight $i$.
\item
Show that each $\mathcal{H}_{\dd}$ is concentrated entirely in even cohomological degree, or entirely in odd cohomological degree, of the predicted parity.
\end{enumerate}
The point is that for pure mixed Hodge structures, the only terms that contribute to the sum in (\ref{cf_exp}) have $n=i$.

We will use an extra geometric assumption to prove a purity statement.  We do not expect it to be essential. 

\begin{theorem}
\label{purity_thm}
Let $G<\SL_3$ be abelian, and such that the exceptional locus of $\pi_G\colon \ghilb{G}\rightarrow \CC^3/G$ is mapped to the origin. Then the mixed Hodge structure $\HHst_G$ is pure, and $\HHst_G$ is concentrated entirely in even or odd degree, depending on whether $(\dd,\dd)$ is even or odd, respectively.
\end{theorem}

We relegate the technical proof of this statement to the last section of our paper. 

\begin{corollary}
Under the same assumption as Theorem \ref{purity_thm}, the mixed Hodge structure on $\HHhilb(a,b,c)$ is pure, and $\HHhilb(a,b,c)_{\dd}$ is concentrated entirely in even or odd degree, depending on whether $(\dd,\dd)+d_0$ is even or odd, respectively.
\end{corollary}
\begin{proof}
From (\ref{WCI2}) there is an inclusion of $\mathbb{N}^{r+1}$-graded mixed Hodge structures 
\[
\HHhilb(a,b,c)_{\dd}\subset \HHst(a,b,c)\otimes^{\tw}\mathbb{Q}_{1_{\infty}}.
\]
The result follows from Theorem \ref{purity_thm} and the definition of the twisted tensor product.
\end{proof}

We deduce

\begin{theorem}\label{thm:positivity}
Let $G<\SL_3$ be abelian, and assume that the exceptional locus of $\pi_G\colon \ghilb{G}\rightarrow \CC^3/G$ is mapped to the origin.  Then
\[
Z^{G(a,b,c)}(\qq^{1/2},\xx)=\sum_{\dd\in\mathbb{N}^r}g_{\dd}(q)(-\qq^{1/2})^{d_0+(\dd,\dd)}\xx^\dd
\]
where $g_{\dd}\in\mathbb{N}[\qq^{\pm 1}]$.
\end{theorem}

\begin{remark}\rm If $G=\mu_r(a,b,c)<\SL_3$ is cyclic, then it is easy to see that the condition that the exceptional locus of $\pi_G\colon \ghilb{G}\rightarrow \CC^3/G$ be mapped to the origin is equivalent to $(a,r)=(b,r)=(c,r)=1$.\end{remark}

\subsection{Proof of Theorem~\ref{purity_thm}}

In this section, the most technical in our paper, we give the proof of the purity result Theorem~\ref{purity_thm}. We begin with a geometric statement, which we used in Section \ref{refin_sec}; the main idea\footnote{Thanks to Tom Bridgeland for explaining this argument to us.} of the proof is exactly the same as the proof of Nakamura's conjecture in \cite{BKR}.

As before, let $(Q,W)$ be the McKay quiver of $G(a,b,c)<\SL(3,\C)$ with its potential (\ref{MP:eq}). Define $\nn=(n,\ldots,n)\in \N^{Q_0}$, a dimension vector on $Q$, $\zeta_0=(r-1,-1,\ldots,-1)\in \Q^{Q_0}$ our fixed stability parameter on $Q$ and $U\subset\Q^{Q_0}$ the small open neighbourhood around it in the space of stability parameters from Proposition~\ref{prop:NI}.

\begin{proposition} \label{ray_prop}Let $\zeta\in U$ be generic. 
The equivalences of categories~$\Phi$ of~\eqref{eq:FM} and its inverse~$\Psi$ induce 
for every positive integer~$n$ an isomorphism of stacks
\[ f_n\colon   \Tors^n(Y_G)\cong \Mst(Q,W)^{\zeta-\sst}_{\nn}
\]
between the moduli stack of $\zeta$-semistable representations of the QP $(Q,W)$ of dimension vector $\nn$ and the stack 
$\Tors^n(\ghilb{G})$ of torsion sheaves on $\ghilb{G}$ of length $n$. If, further, the exceptional locus of $\pi_G\colon \ghilb{G}\rightarrow \CC^3/G$ maps to the origin, then~$f_n$ restricts to an isomorphism of groupoids
\[ f_{n, \nilp}\colon   (\Tors^n(\ghilb{G})(\CC))_{\exc}\cong (\Mst(Q,W)^{\zeta-\sst}_{\nn}(\CC))_{\nilp}
\]
where on the right we restrict to nilpotent $\CC Q$-modules, and on the left we restrict to sheaves set-theoretically supported on the exceptional locus of $\pi_G\colon\ghilb{G}\rightarrow\CC^3/G$.
\end{proposition} 
\begin{proof}  Since the equivalence $\Phi$ is derived from a Fourier--Mukai kernel $\mathcal{Q}$, there is a canonical way to try to upgrade it to a morphism of stacks --- we take the Fourier--Mukai kernel $\mathcal{Q}\boxtimes \mathcal{O}_S$ to define a functor from $S$-flat families of coherent sheaves on $S\times \ghilb{G}$ to complexes of $\mathcal{O}_S\otimes \CC(Q,W)$-modules on $S$.  We need to show that the complexes we obtain this way are supported in cohomological degree zero and are moreover $S$-flat families of $\zeta$-semistable modules.  By \cite[Lem.4.3]{Bri} this statement follows from the claim that this construction gives an isomorphism of groupoids of $K$-points for any $K\supset \CC$, which furthermore proves that the morphism of stacks is an isomorphism.

For simplicity we assume that $K=\CC$, as for a general field the proofs are unchanged.  For $n=1$, we have $\Tors^n(Y_G)\cong Y_G$,
so the result follows from Proposition~\ref{prop:NI}. 
We prove the general case by induction on $n$. First take $\FF\in \Tors^n(\ghilb{G})$. Any such sheaf is a finite
iterated extension of point sheaves $\OO_{y_i}$ for (not necessarily distinct) points $y_i\in Y$, 
and in particular for some $y\in Y_G$ there exists an exact sequence
\[ 0 \to  \OO_y \to \FF \to \FF' \to 0.\]
The sheaf $\FF'$ is then a torsion sheaf of length $n-1$, and we get a distinguished triangle
\[ \Phi(\OO_y)\to \Phi(\FF)\to \Phi(\FF')\stackrel{[1]}\longrightarrow
\]
in $\D(\CC(Q,W)\lmod)$. By induction, $\Phi(\FF')$ and $\Phi(\OO_y)$ are $\zeta$-semistable
representations of $(Q,W)$, of dimension vectors ${\bf n-1}$ and ${\bf 1}$ respectively. Hence $\Phi(\FF)$ is also an actual representation,  and also $\zeta$-semistable. This gives a morphism 
\[  g_n\colon \Tors^n(Y)(\CC) \to \Mst(Q,W)^{\zeta-\sst}_{\nn}(\CC) 
\]
that is injective on isomorphism classes and an isomorphism on morphism spaces since $\Phi$ is an equivalence of categories. 

To show that $g_n$ is surjective, we follow~\cite[Section 8]{BKR}. Suppose that $M$ is a $\zeta$-stable $\CC(Q,W)$-representation of dimension vector $\nn$, for $n>1$. Then for all $\zeta$-(semi)stable representations $N$ 
of dimension vector $(1,1,1)$, we have $\Hom(M,N)=\Hom(N,M)=0$, so $\Ext^i(M,N)$ lives in degrees
$1$ and $2$ only. Thus $\Psi(M)$ is a 2-term complex of vector bundles $[L^{-2}\xrightarrow{d}
L^{-1}]$ on $Y_G$.  The cohomology of this complex is compactly supported, and so $d$ is
injective. So $\Psi(M)$ is a coherent sheaf on $Y$ living in degree~$(-1)$.  This sheaf cannot have a summand supported away from the exceptional locus $E$, since the corresponding summand of $M$ would have negative dimension vector.  So $\Psi(M)[1]$ is set-theoretically supported on $E$, and thus has a filtration by coherent sheaves on $E$.  The Chern characters of these sheaves sum to $-n$ times the chern Character of the skyscraper sheaf of a point, which is a contradiction. Hence any semistable representation of dimension vector $\nn$ is strictly semistable and in particular an iterated extension of stable representations of dimension vector $\dd_0$. The converse of the inductive argument above then shows that its image under $\Psi$ is a torsion sheaf on $Y$. Thus $g_n$ is an isomorphism of groupoids as required.

Finally, suppose that $\FF$ in the above argument is supported along the exceptional locus.  By our assumption on $\pi$, $\mathcal{F}$ is annihilated by all monomials $x^ry^sz^t$, considered as functions on $Y_G$, for $r+s+t\gg 0$.  As such, $\Phi(\mathcal{F})$ is annihilated by multiplication by the central elements 
\[\left(\sum_{0\leq i\leq r-1}x_i\right)^r\left(\sum_{0\leq i\leq r-1}y_i\right)^s
\left(\sum_{0\leq i\leq r-1}z_i\right)^t
\]
in $\mathbb{C}(Q,W)$, for $r+s+t\gg0$ i.e. it is nilpotent.  The converse works in the same way. \end{proof}

\begin{proof}[Proof of Theorem~\ref{purity_thm}]
We introduce an auxiliary mixed Hodge structure 
\[
\HHst(a,b,c)^{\znilp}=\bigoplus_{\dd\in\mathbb{N}^{Q_0}}\HO_c\left(\Mst_{\dd}(Q),\phi_{\Tr(W)_\dd}\lvert_{\znilp}\right)
\]
by restricting the vanishing cycle mixed Hodge module to the reduced substack, the $\CC$-points of which are defined by the condition that the action of $z\in\CC[x,y,z]\rtimes G$ is nilpotent.  We will prove that the purity and parity statements of the theorem are true of both $\HHst(a,b,c)$ and $\HHst(a,b,c)^{\znilp}$.  By \cite[Theorem B]{DaMe} we have the following variant of (\ref{WCI1}):
\begin{align}
\label{Lndecomp}
\HHst(a,b,c)^{\znilp}\cong &\bigotimes_{-\infty\xrightarrow{\theta} \infty}^{\tw} \HO_c\left(\Mst_{\theta}^{\sst}(Q),\phi_{\Tr(W)}\lvert_{\znilp}\right).
\end{align}
The parity and purity statements for the left hand side of (\ref{WCI1}) and (\ref{Lndecomp}) are equivalent to the same parity and purity statements for the components of the tensor product decompositions.  This is because purity is preserved by tensor product, $\LL^{1/2}$ is pure, and the degree shift introduced by the twisted monoidal product is equal to the part of the following equation that we have put in square brackets:
\begin{align*}
(\dd,\dd)+(\ee,\ee)+\left[(\ee,\dd)-(\dd,\ee)\right]=(\dd+\ee,\dd+\ee)\ \mod 2.
\end{align*}

Under the derived equivalence $\Phi$, complexes of $\CC(Q,W)$-modules with finite-dimensional cohomology correspond to complexes of coherent sheaves with compactly supported cohomology sheaves, and nilpotency for the cohomology of $\CC(Q,W)$-modules corresponds to cohomology sheaves being set-theoretically supported on the exceptional locus of $\pi_G\colon \ghilb{G}\rightarrow \CC^3/G$.  

Without affecting the definition of $\zeta$-stability we may adjust our stability condition $\zeta$ so that the dimension vector $\dd_0=(1,\ldots,1)$ has slope zero, by adding some scalar multiple of $(1,\ldots,1)$ to $\zeta$.  

Let $\dd$ be a dimension vector of nonzero slope with respect to $\zeta$.  If $M$ is $\zeta$-semistable and $\dd$-dimensional, the sheaf $\Psi(M)$ must be supported set-theoretically on the exceptional locus of $\pi_G$, for if $\mathcal{G}\subset \Psi(M)$ is the summand supported away from the exceptional locus, then $\dim(\supp(\Phi(\mathcal{G})))=0$ and so $\Phi(\mathcal{G})$ is a direct summand of $\rho$ with slope zero.  It follows that the support of $\phi_{\Tr(W)_{\dd}}$, restricted to the semistable locus, only contains nilpotent modules, and in particular, modules for which the action of $z$ is nilpotent.  So for $\theta\neq 0$ there is an equality
\begin{equation}
\label{ne}
\HO_c\left(\Mst_{\theta}^{\sst}(Q),\phi_{\Tr(W)}\right)=\HO_c\left(\Mst_{\theta}^{\sst}(Q),\phi_{\Tr(W)}\lvert_{\znilp}\right).
\end{equation}

If we let $\CC^*$ act on the edges of $Q$ by scaling the edges $x_1,\ldots,x_r$ and leaving the other edges invariant, the function $\Tr(W)$ is a weight one eigenfunction for $\CC^*$, and so as a special case of the dimensional reduction isomorphism\footnote{I.e. we ``reduce'' along the $x$ direction, from a categorically three-dimensional moduli problem to a two-dimensional one.} we have
\begin{align*}
\HHst(a,b,c)_{\dd}^{\znilp}\cong &\HO_c(\mathfrak{R}_\dd,\mathbb{Q})\otimes\LL^{(\dd,\dd)/2+r},
\end{align*}
where $r=\sum_{0\leq i\leq r-1} d_i d_{i+c}$ and
\begin{align*}
\mathfrak{R}_{\dd}=&\Rep_{\dd}^{\znilp}(\CC[y,z]\rtimes G),
\end{align*}
is the stack of $\CC[y,z]\rtimes G$-modules for which $z$ acts nilpotently, and the underlying $G$-representation is given by
\[
\bigoplus_{0\leq i\leq r-1}\rho_i^{\oplus d_i}.
\]
Because of the constraint on $z$, for a given $\dd$ there are finitely many possibilities for the isomorphism class of the underlying $\CC[z]\rtimes G$-representation.  We can partition the stack $\mathfrak{R}_\dd$ according to the isomorphism type of this representation, and it is straightforward to check as in \cite{Dav2} that each piece of this partition has pure compactly supported cohomology, supported entirely in even degree.  From the resulting  long exact sequences in compactly supported cohomology, it follows that $\mathfrak{R}_{\dd}$ has pure compactly supported cohomology, supported in even degree.  It follows that $\HHst(a,b,c)^{\znilp}_G$ is pure, supported in degrees of the required parity, and so the same is true of the terms in (\ref{WCI1}), for $\theta\neq 0$, as well as the right hand side of (\ref{Lndecomp}) for all $\theta$.

Let $\theta=0$, so that by Proposition \ref{ray_prop} the only semistable $\CC(Q,W)$-modules of slope $\theta$ have dimension vector $\nn$ for some $n$ and correspond under the derived equivalence $\Phi$ to coherent sheaves with zero-dimensional support.  As in \cite{Dav3} we can write
\begin{align}
\label{Int}
\HO_c\left(\Mst_{\theta}^{\sst}(Q),\phi_{\Tr(W)}\right)\cong&\Sym\left(\bigoplus_{n\geq 1} \HO_c(\ghilb{G},\mathcal{L}_n)\otimes\HO_c(\pt/\CC^*)\otimes\LL^{1/2}\right)\\
\label{nInt}
\HO_c\left(\Mst_{\theta}^{\sst}(Q),\phi_{\Tr(W)}\lvert_{\nilp}\right)\cong&\Sym\left(\bigoplus_{n\geq 1} \HO_c(\ghilb{G},\mathcal{L}_n\lvert_{\nilp})\otimes\HO_c(\pt/\CC^*)\otimes\LL^{1/2}\right)
\end{align}
where $\mathcal{L}_n$ is a pure rank one variation of mixed Hodge structure on $\ghilb{G}$, and $\mathcal{L}_n\lvert_{\nilp}$ is its restriction to the locus where $z$ acts nilpotently.  Since $(\nn,\nn)=0$, we need to show that the left hand side of (\ref{Int}) is pure, of even degree, which is equivalent to showing that each $\HO_c(\ghilb{G},\mathcal{L}_n)$ is pure of odd degree, because of the extra twist by $\LL^{1/2}$ on the right hand side.  

We claim that in fact $\mathcal{L}_n\lvert_{\nilp}\cong \underline{\mathbb{Q}}_{Y_G}\otimes\LL^{-3/2}$.  This follows from the (analytic) local triviality of this mixed Hodge module, along with the claim that $\pi_1(\ghilb{G})=1$.  For this claim, note that by \cite{Nak}, $\ghilb{G}$ is a smooth toric variety, and since it has Euler characteristic~$r$, the associated fan contains a 3-dimensional cone.  The claim then follows from 
\cite[Thm 12.1.10]{CLS}.  Finally, then, we have to prove that either of the isomorphic mixed Hodge structures
\[
\HO_c(\ghilb{G},\mathbb{Q})\otimes\LL^{-3/2}\cong \left(\HO(\ghilb{G},\mathbb{Q})\otimes\LL^{-3/2}\right)^*
\]
are pure, with cohomology in entirely odd degree.
The purity, of weight $i$, of $\HO^i(\ghilb{G},\mathbb{Q})$ follows by a standard argument: the $m$th graded piece of the weight filtration is zero for $m<i$ since $\ghilb{G}$ is smooth, on the other hand, $Y_G$ contracts onto its exceptional locus $E=\pi_G^{-1}(0)$, and the $m$th graded piece of the weight filtration on $\HO^i(E,\mathbb{Q})$ is zero for $m>i$ since $E$ is proper.  

For the parity statement, from the composition of isomorphisms
\[
\HO(\ghilb{G},\mathbb{Q})\cong \HO(E,\mathbb{Q})\cong \HO_c(E,\mathbb{Q})
\]
we deduce that it is enough to show that $\HO_c(E,\mathbb{Q})$ is concentrated in even degree.  Let $Z\subset \ghilb{G}$ be the locus cut out by the equation $z=0$.  Then the complement $U=Z\setminus E$ admits a Galois cover from $\mathbb{C}^2\setminus \{0\}$, i.e. the locus inside $\mathbb{C}^3$ for which $z=0$ and one of $x,y$ are nonzero.  In particular, the only impure compactly supported cohomology of $U$ is concentrated in odd cohomological degree, as the same is true of $\mathbb{C}^2\setminus \{0\}$.  Moreover, the compactly supported cohomology of both $Z$ and $E$ are pure -- purity for $Z$ follows from purity of the left hand side of (\ref{nInt}) and our description of $\mathcal{L}_n$.  Then since the morphisms in the long exact sequence in cohomology
\[
\rightarrow \HO^i_c(U,\mathbb{Q})\rightarrow \HO^i_c(Z,\mathbb{Q})\rightarrow \HO^i_c(E,\mathbb{Q})\rightarrow
\]
are morphisms of mixed Hodge structures it follows that $\HO_c(E,\mathbb{Q})$ is concentrated in even degree.
\end{proof}

\newpage 

\section*{Appendix}

Here we record the coefficients of the reduced $\mu_3(1,1,1)$ partition function up to total degree $24$.  

{\scriptsize

\[\begin{split}Z^{\mu_3(1,1,1)}_\red(t_0,t_1,t_2)&=1+t_0+3t_0t_1+3t_0t_1^2+t_0t_1^3+9t_0t_1^2t_2+6t_0t_1^3t_2+9t_0t_1^2t_2^2+
9t_0^2t_1^2t_2^2+15t_0t_1^3t_2^2+
\\& 3t_0t_1^2t_2^3+12t_0^2t_1^3t_2^2+12t_0^2t_1^2t_2^3+20t_0t_1^3t_2^3+18t_0^3t_1^2t_2^3+
46t_0^2t_1^3t_2^3+15t_0t_1^3t_2^4+12t_0^4t_1^2t_2^3+
\\&44t_0^3t_1^3t_2^3+66t_0^2t_1^3t_2^4+
6t_0t_1^3t_2^5+3t_0^5t_1^2t_2^3+36t_0^4t_1^3t_2^3+15t_0^3t_1^4t_2^3+114t_0^3t_1^3t_2^4+42t_0^2t_1^3t_2^5+
\\&t_0t_1^3t_2^6+18t_0^5t_1^3t_2^3+36t_0^4t_1^4t_2^3+
96t_0^4t_1^3t_2^4+45t_0^3t_1^4t_2^4+126t_0^3t_1^3t_2^5+
10t_0^2t_1^3t_2^6+45t_0^5t_1^4t_2^3+
\\&12t_0^4t_1^5t_2^3+39t_0^5t_1^3t_2^4+153t_0^4t_1^4t_2^4+210t_0^4t_1^3t_2^5+
45t_0^3t_1^4t_2^5+45t_0^3t_1^3t_2^6+60t_0^5t_1^5t_2^3+\\&6t_0^6t_1^3t_2^4+144t_0^5t_1^4t_2^4+54t_0^4t_1^5t_2^4+
210t_0^5t_1^3t_2^5+234t_0^4t_1^4t_2^5+120t_0^4t_1^3t_2^6+15t_0^3t_1^4t_2^6+\\&45t_0^5t_1^6t_2^3+36t_0^6t_1^4t_2^4+
225t_0^5t_1^5t_2^4+126t_0^6t_1^3t_2^5+486t_0^5t_1^4t_2^5+90t_0^4t_1^5t_2^5+210t_0^5t_1^3t_2^6+\\&120t_0^4t_1^4t_2^6+
18t_0^5t_1^7t_2^3+90t_0^6t_1^5t_2^4+201t_0^5t_1^6t_2^4+42t_0^7t_1^3t_2^5+504t_0^6t_1^4t_2^5+468t_0^5t_1^5t_2^5+
\\&252t_0^6t_1^3t_2^6+420t_0^5t_1^4t_2^6+66t_0^4t_1^5t_2^6+3t_0^5t_1^8t_2^3+120t_0^6t_1^6t_2^4+108t_0^5t_1^7t_2^4+
6t_0^8t_1^3t_2^5+\\&261t_0^7t_1^4t_2^5+354t_0^5t_1^6t_2^5+846t_0^6t_1^5t_2^5+210t_0^7t_1^3t_2^6+840t_0^6t_1^4t_2^6+
429t_0^5t_1^5t_2^6+18t_0^4t_1^5t_2^7+\\&90t_0^6t_1^7t_2^4+27t_0^5t_1^8t_2^4+54t_0^8t_1^4t_2^5+684t_0^7t_1^5t_2^5+
828t_0^6t_1^6t_2^5+270t_0^5t_1^7t_2^5+120t_0^8t_1^3t_2^6+\\&1050t_0^7t_1^4t_2^6+1296t_0^6t_1^5t_2^6+306t_0^5t_1^6t_2^6+
126t_0^5t_1^5t_2^7+36t_0^6t_1^8t_2^4+216t_0^8t_1^5t_2^5+975t_0^7t_1^6t_2^5+\\&558t_0^6t_1^7t_2^5+108t_0^5t_1^8t_2^5+
45t_0^9t_1^3t_2^6+840t_0^8t_1^4t_2^6+2319t_0^7t_1^5t_2^6+1420t_0^6t_1^6t_2^6+360t_0^5t_1^7t_2^6+\\&378t_0^6t_1^5t_2^7+
129t_0^5t_1^6t_2^7+ 6 t_0^6 t_1^9 t_2^4  + 504 t_0^8 t_1^6 t_2^5+ 810 t_0^7 t_1^7 t_2^5  + 
 252 t_0^6 t_1^8 t_2^5 + 10 t_0^{10} t_1^3 t_2^6  + 
\\& 420 t_0^9 t_1^4 t_2^6  + 2586 t_0^8 t_1^5 t_2^6  + 3207 t_0^7 t_1^6 t_2^6 + 
 1284 t_0^6 t_1^7 t_2^6  + 252 t_0^5 t_1^8 t_2^6  + 
 630 t_0^7 t_1^5 t_2^7  +852 t_0^6 t_1^6 t_2^7 + 
 \\&270 t_0^5 t_1^7 t_2^7  + 21 t_0^5 t_1^6 t_2^8  + 756 t_0^8 t_1^7 t_2^5 + 387 t_0^7 t_1^8 t_2^5  + 
  54 t_0^6 t_1^9 t_2^5 + t_0^{11} t_1^3 t_2^6  + 120 t_0^{10} t_1^4 t_2^6  + 
 1755 t_0^9 t_1^5 t_2^6  + \\&4668 t_0^8 t_1^6 t_2^6  +  3138 t_0^7 t_1^7 t_2^6  +
  795 t_0^6 t_1^8 t_2^6  + 
 630 t_0^8 t_1^5 t_2^7  + 2355 t_0^7 t_1^6 t_2^7  + 
  1404 t_0^6 t_1^7 t_2^7  + 378 t_0^5 t_1^8 t_2^7  + 
  \\& 147 t_0^6 t_1^6 t_2^8  +108 t_0^5 t_1^7 t_2^8  + 756 t_0^8 t_1^8 t_2^5  + 
  96 t_0^7 t_1^9 t_2^5  + 15 t_0^{11} t_1^4 t_2^6  + 660 t_0^{10} t_1^5 t_2^6  + 
 4320 t_0^9 t_1^6 t_2^6  + \\& 5622 t_0^8 t_1^7 t_2^6  + 
  2022 t_0^7 t_1^8 t_2^6  + 216 t_0^6 t_1^9 t_2^6  + 
 378 t_0^9 t_1^5 t_2^7  + 3480 t_0^8 t_1^6 t_2^7  + 
 4140 t_0^7 t_1^7 t_2^7  + 1494 t_0^6 t_1^8 t_2^7  + \\&
 441 t_0^7 t_1^6 t_2^8  + 738 t_0^6 t_1^7 t_2^8  + 378 t_0^5 t_1^8 t_2^8  + 
 18 t_0^5 t_1^7 t_2^9  + 504 t_0^8 t_1^9 t_2^5  +  9 t_0^7 t_1^{10} t_2^5  + 
 105 t_0^{11} t_1^5 t_2^6  + \\&2200 t_0^{10} t_1^6 t_2^6  + 
 6930 t_0^9 t_1^7 t_2^6  + 4920 t_0^8 t_1^8 t_2^6  + 
 693 t_0^7 t_1^9 t_2^6  + 126 t_0^{10} t_1^5 t_2^7  + 
 2895 t_0^9 t_1^6 t_2^7  + 8172 t_0^8 t_1^7 t_2^7  + 
 \\&4365 t_0^7 t_1^8 t_2^7  + 504 t_0^6 t_1^9 t_2^7  + 
 735 t_0^8 t_1^6 t_2^8  + 2412 t_0^7 t_1^7 t_2^8  + 
 1845 t_0^6 t_1^8 t_2^8  + 150 t_0^6 t_1^7 t_2^9  + 
 252 t_0^5 t_1^8 t_2^9  + \\&216 t_0^8 t_1^{10} t_2^5  + 
 455 t_0^{11} t_1^6 t_2^6  + 4950 t_0^{10} t_1^7 t_2^6  + 
 7560 t_0^9 t_1^8 t_2^6  + 3270 t_0^8 t_1^9 t_2^6  + 
 81 t_0^7 t_1^{10} t_2^6  + 18 t_0^{11} t_1^5 t_2^7  + 
 \\&1284 t_0^{10} t_1^6 t_2^7  + 9612 t_0^9 t_1^7 t_2^7  + 
 10800 t_0^8 t_1^8 t_2^7  + 2220 t_0^7 t_1^9 t_2^7  + 
 735 t_0^9 t_1^6 t_2^8  + 4680 t_0^8 t_1^7 t_2^8  + 
 \\&5346 t_0^7 t_1^8 t_2^8  + 756 t_0^6 t_1^9 t_2^8  + 
 546 t_0^7 t_1^7 t_2^9  + 1536 t_0^6 t_1^8 t_2^9  + 108 t_0^5 t_1^8 t_2^{10} + 
  +54 t_0^8 t_1^{11} t_2^5  + 
  \\&1365 t_0^11 t_1^7 t_2^6  + 
 7920 t_0^{10} t_1^8 t_2^6  + 5670 t_0^9 t_1^9 t_2^6  + 
 1596 t_0^8 t_1^{10} t_2^6  + 237 t_0^11 t_1^6 t_2^7  + 
 5886 t_0^{10} t_1^7 t_2^7  + 
 \\& 18045 t_0^9 t_1^8 t_2^7  + 
 9114 t_0^8 t_1^9 t_2^7  + 324 t_0^7 t_1^10 t_2^7  + 
 441 t_0^{10} t_1^6 t_2^8  + 5580 t_0^9 t_1^7 t_2^8  + 
 11934 t_0^8 t_1^8 t_2^8  + \\& 4143 t_0^7 t_1^9 t_2^8 + 
 1134 t_0^8 t_1^7 t_2^9  + 4497 t_0^7 t_1^8 t_2^9  + 
 756 t_0^6 t_1^9 t_2^9  + 837 t_0^6 t_1^8 t_2^{10}  + 27 t_0^5 t_1^8 t_2^{11}
 +\\&  O((t_0,t_1,t_2)^{25}).
\end{split}\]
}

\end{document}